\documentclass[11pt,a4paper]{article}
\usepackage{amsmath}
\usepackage{a4wide}
\usepackage{mathrsfs}
\usepackage{stmaryrd}
\usepackage{pifont}
\usepackage{amssymb}
\usepackage{color}
\usepackage{amsthm}
\usepackage{graphicx}
\usepackage[colorlinks=true]{hyperref}
\hypersetup{linkcolor=blue, urlcolor=red, citecolor=red}
\allowdisplaybreaks[4] \numberwithin{equation}{section}
\newtheorem{theorem}{Theorem}[section]

\newtheorem{lemma}{Lemma}[section]
\newtheorem{proposition}{Proposition}[section]
\newtheorem{remark}{Remark}[section]

\newtheorem{corollary}[theorem]{Corollary}
 \input amssym.def
\input amssym
\title{Approximation of time optimal controls for heat equations with perturbations in the system potential}
\author{Huaiqiang Yu\thanks{School of Mathematics and Statistics, Wuhan University,
    Wuhan, 430072, P. R. China. Email: huaiqiangyu@whu.edu.cn}}
\date{}
\begin{document}
\maketitle
\begin{abstract}
    In this paper, we study a certain approximation property for  a time optimal control problem of the heat equation with $L^\infty$-potential. We prove that the optimal time and the optimal control to the  same time optimal control problem for the heat equation,  where the potential has a small perturbation, are close to those for the original problem. We also verify that for the heat equation with a small perturbation in the potential, one can construct a new time optimal control problem, which has the same target as  that of the original problem,  but has a different control constraint bound from that of the original problem, such that the new and the original  problems share the same optimal time, and meanwhile the optimal control of the new problem is close to that of the original one.
     The main idea to approach such approximation is an appropriate  use of an equivalence theorem
    of minimal norm and minimal time control problems for the heat equations under consideration.
    This theorem was first established by G.Wang and E. Zuazua in \cite{b4} for the case where the controlled system is an internally controlled heat equation without the potential and the target is  the origin of the state space.

\end{abstract}
\section{Introduction}
    Let $\Omega\subset\mathbb{R}^N$ be a bounded domain with a
    smooth boundary $\partial\Omega$ and $\omega$ be an open and
    nonempty subset of $\Omega$. Denote by $\chi_\omega$ the
    characteristic function of the set $\omega$. Write
    $\mathbb{R}^+=(0,+\infty)$.
Consider the following controlled heat
    equations:
\begin{equation}\label{equationyu2.01}
\begin{cases}
     y_t-\triangle y-ay=\chi_\omega
    u&\mbox{in}\;\;\Omega\times\mathbb{R}^+,\\
    y=0&\mbox{on}\;\;\partial\Omega\times \mathbb{R}^+,\\
    y(0)=y_0&\mbox{in}\;\;\Omega
\end{cases}
\end{equation}
    and
\begin{equation}\label{equationyu2.02}
\begin{cases}
    y^\varepsilon_t-\triangle y^\varepsilon-a_\varepsilon y^\varepsilon=\chi_\omega
    u&\mbox{in}\;\;\Omega\times\mathbb{R}^+,\\
    y^\varepsilon=0&\mbox{on}\;\;\partial\Omega\times \mathbb{R}^+,\\
    y^\varepsilon(0)=y_0&\mbox{in}\;\;\Omega,
\end{cases}
\end{equation}
    where $y_0$ is in $L^2(\Omega)$, $u$ is a control taken from
    the space $L^\infty(\mathbb{R}^+;L^2(\Omega))$, $a$ and $ a_\varepsilon$, with $\varepsilon>0$ small,  belong to $L^\infty(\Omega)$. Here, we assume

   \vskip 5pt

   \noindent  $(H_1)$   $\|a_\varepsilon-a\|_{L^\infty(\Omega)}\to
    0$ as $\varepsilon\to 0^+$.

    \vskip 5pt

\vskip 5pt

    \noindent  $(H_2)$    $y_0\in L^2(\Omega)$  such that
    $y_0\notin \overline{B_K(0)}$, where $\overline{B_K(0)}$ is the closed ball in $L^2(\Omega)$, centered at the origin and of radius $K>0$.

    \vskip 5pt

    \noindent  $(H_3)$  Either  $\|a\|_{L^\infty(\Omega)}<\lambda_1$
    or $a(x)\leq0$ for any
    $x\in\Omega$,
    where $\lambda_1>0$ is the first eigenvalue to the operator
    $-\triangle$ with the domain  $D(\triangle)=H_0^1(\Omega)\cap H^2(\Omega)$.

    \vskip 5pt

\noindent
    Corresponding to each $u$ and  $y_0$,  the equations
    (\ref{equationyu2.01}) and (\ref{equationyu2.02}) have unique solutions which will be treated as functions of time variable $t$, from $[0,+\infty)$
    to the space $L^2(\Omega)$ and denoted by
       $y(\cdot;u,y_0)$ and $y^{\varepsilon}(\cdot;u,y_0)$ respectively.
    One can easily check that, under the assumption $(H_1)$,
\begin{equation}\label{yu2.03}
    \|y^{\varepsilon}(\cdot;u,y_0)-y(\cdot;u,y_0)\|_{C([0,T];L^2(\Omega))}\to
    0\;\;\mbox{as}\;\;\varepsilon\to 0^+,
\end{equation}
    when $T>0$, $y_0\in L^2(\Omega)$ and $u\in
    L^\infty(0,T;L^2(\Omega))$.

\par
    We start with introducing some notations which will be used in this paper frequently.
     Denote by $\|\cdot\|_\Omega$ and $\langle\cdot,\cdot\rangle_\Omega$ the
    usual norm and inner product of  the space $L^2(\Omega)$ respectively. Write accordingly
    $\|\cdot\|_\omega$ and $\langle\cdot,\cdot\rangle_\omega$
    for the norm and inner product of the space $L^2(\omega)$. Use
    $\overline{B_r(0)}$ to denote the closed ball in $L^2(\Omega)$,
    centered at zero point and of radius $r>0$. When $X$ is a Banach space, $\|\cdot\|_X$ stands for the norm of $X$ and
    $\|\cdot\|$ denotes the standard operator norm over $\mathcal{L}(X)$ which is the space of all linear and bounded operators on $X$.

    Next, we fix two positive numbers $K$ and $M$,  choose the target set $ \overline{B_K(0)}$ in $L^2(\Omega)$ and
    define two constraint sets of controls as follows:

    \vskip 5pt

\noindent  $\mathcal{U}_{M}\equiv\{u\in L^\infty(\mathbb{R}^+;L^2(\Omega)):
    u(\cdot)\in
    \overline{B_M(0)}\;\mbox{over}\;\mathbb{R}^+\;\mbox{and}\;
    \exists\;t>0\;\mbox{s.t.}\; y(t;u,y_0)\in \overline{B_K(0)}\};$

    \vskip 5pt

\noindent
   $\mathcal{U}^{\varepsilon}_{M}\equiv\{u\in
    L^\infty(\mathbb{R}^+;L^2(\Omega)):
    u(\cdot)\in
    \overline{B_M(0)}\;\mbox{over}\;\mathbb{R}^+\;\mbox{and}\;
    \exists \; t>0\;\mbox{s.t.}\; y^{\varepsilon}(t;u,y_0)\in
    \overline{B_K(0)}\}.$

    \vskip 5pt

   \noindent  Now, we set up  the following two time optimal control problems:

    \vskip 5pt

   \noindent   $(TP)\;\;\;\; T^* \equiv \inf_{u\in\mathcal{U}_{M}}\{
    t\in\mathbb{R}^+:y(t;u,y_0)\in\overline{B_K(0)}\};$

    \vskip 5pt

   \noindent  $(TP^\varepsilon_1)\;\;\;\; T^{*,1}_\varepsilon
    \equiv\inf_{u\in\mathcal{U}^{\varepsilon}_{M}}\{
    t\in\mathbb{R}^+:y^\varepsilon(t;u,y_0)\in\overline{B_K(0)}\}.$

    \vskip 5pt

\noindent The numbers  $T^*$ and $T^{*,1}_\varepsilon$
    are called the optimal time
    for the problems $(TP)$  and $(TP^\varepsilon_1)$ respectively.  A control $u^*\in \mathcal{U}_{M}$ is called an optimal control to
    $(TP)$ if $y(T^*; u^*,y_0)\in \overline{B_K(0)}$ and $u^*(\cdot)=0$ over $(T^*, +\infty)$. An optimal control $u^{*,1}_{\varepsilon}$
    to $(TP^\varepsilon_1)$ is defined in a similar way.

    \par

   The first purpose of this paper is to study the convergence of the problem $( TP^\varepsilon_1)$ to the problem  $(TP)$ as $\varepsilon$ tends to zero. The results are included in the following theorem:

\begin{theorem}\label{ithyu1.1}
    Suppose that $(H_1)$, $(H_2)$ and $(H_3)$ hold. Let $T^*$, $u^*$ and  $T^{*,1}_{\varepsilon}$, $u^{*,1}_\varepsilon$  be  the optimal time and the optimal controls  to
    Problems $(TP)$ and $(TP^\varepsilon_1)$ respectively.
    Then

   \noindent  $(i)$ $T^{*,1}_\varepsilon\to T^*$ as $\varepsilon\to 0^+$;

   \noindent  $(ii)$ $u^{*,1}_\varepsilon\to u^*$ strongly in
    $L^2((0,T^*)\times\Omega)$ as $\varepsilon\to 0^+$;

    \noindent $(iii)$ for any $\eta\in(0,T^*)$, $u^{*,1}_\varepsilon\to u^*$
    strongly in $L^\infty(0,T^*-\eta;L^2(\Omega))$ as
    $\varepsilon\to 0^+$.
\end{theorem}
\par

    It is not hard to show the convergence of the optimal time. However, it is not trivial to prove the above-mentioned $L^\infty$-convergence of the optimal controls. We make use of
    an equivalence theorem of minimal time and minimal norm control problems,  and the convergence of the associated minimization problems (which will be introduced later),  as well as the bang-bang property to reach the aim.  The equivalence theorem (see Proposition \ref{lemmayu4.2}) is a slight modified version of that established in  \cite{b4} (see Theorem 1.1 in \cite{b4}), while the bang-bang property was built up in Theorem 1 of \cite{b10}.   To state the associated minimization problems, we first consider two equations as follows:
\begin{equation}\label{103}
\begin{cases}
    \varphi_t+\triangle\varphi+a\varphi=0&\mbox{in}\;\;\Omega\times(0,T^*),\\
    \varphi=0&\mbox{on}\;\;\partial\Omega\times(0,T^*),\\
    \varphi(T^*)=\varphi_{T^*}\in L^2(\Omega)
\end{cases}
\end{equation}
    and
\begin{equation}\label{104}
\begin{cases}
    \varphi^\varepsilon_t+\triangle\varphi^\varepsilon
    +a_\varepsilon \varphi^\varepsilon=0
    &\mbox{in}\;\;\Omega\times(0,T^{*, 1}_\varepsilon),\\
    \varphi^\varepsilon=0&\mbox{on}\;\;
    \partial\Omega\times(0,T^{*,1}_\varepsilon),\\
    \varphi^\varepsilon(T^{*,1}_\varepsilon)
    =\varphi^\varepsilon_{T^{*,1}_\varepsilon}\in L^2(\Omega).
\end{cases}
\end{equation}
    Write $\varphi(\cdot;\varphi_{T^*},T^*)$ and $\varphi^\varepsilon
    (\cdot;\varphi^\varepsilon_{T^{*,1}_\varepsilon},T^{*,1}_\varepsilon)$ for the solutions of equation (\ref{103}) and equation (\ref{104}) respectively. Then, we set up two functionals over $L^2(\Omega)$ by
\begin{eqnarray}\label{101}
    J^{T^*}(\varphi_{T^*})&=&\frac{1}{2}\left(\int_0^{T^*}\|\varphi(t;\varphi_{T^*},T^*)\|_\omega
     dt\right)^2+\langle
    y_0,\varphi(0;\varphi_{T^*},T^*)\rangle_\Omega\nonumber\\
    &\;&+K\|\varphi_{T^*}\|_\Omega,\;\;
    \varphi_{T^*}\in L^2(\Omega)
\end{eqnarray}
    and
\begin{eqnarray}\label{102}
    J^{T^{*,1}_\varepsilon}_\varepsilon(\varphi^\varepsilon_{T^{*,1}_\varepsilon})
    &=&\frac{1}{2}\left(\int_0^{T^{*,1}_\varepsilon}
    \|\varphi^\varepsilon(t;\varphi^\varepsilon_
    {T^{*,1}_\varepsilon},T^{*,1}_\varepsilon)\|_{\omega}dt\right)^2+\langle
    y_0,\varphi^\varepsilon(0;\varphi_{T^{*,1}_\varepsilon}^\varepsilon,
    T^{*,1}_\varepsilon)\rangle_\Omega\nonumber\\
    &\;&+K\|\varphi^\varepsilon_{T^{*,1}_\varepsilon}\|_\Omega,\;\;
    \varphi^\varepsilon_{T^{*,1}_\varepsilon}\in L^2(\Omega).
\end{eqnarray}
   Now the associated minimization problems are to minimize accordingly  $J^{T^*}(\cdot)$ and $J^{T^{*,1}_\varepsilon}_\varepsilon(\cdot)$ over $L^2(\Omega)$.
   These two minimization problems have unique solutions   $\hat{\varphi}_{T^*}$ and $\hat{\varphi}_{T^{*,1}_\varepsilon}^\varepsilon$  respectively (see Section 4.2 in \cite {b3}).
    With the aid of the above-mentioned equivalence theorem, we can explicitly express the time optimal controls $u^*$   over $[0,T^*)$ and
    $u^{*,1}_\varepsilon$ over  $[0,T^{*,1}_\varepsilon)$ by
\begin{equation}\label{yu0.01}
    u^*(t)=M\frac{\chi_\omega\varphi
    (t;\hat{\varphi}_{T^*},T^*)}{\|\varphi(t;\hat{\varphi}_{T^*},T^*)\|_{\omega}}\;\;\;\;
    \mbox{for all}\;\ t\in[0,T^*),
\end{equation}
\begin{equation}\label{yu0.02}
    u^{*,1}_\varepsilon(t)
    =M\frac{\chi_\omega\varphi^\varepsilon(t;
    \hat{\varphi}^\varepsilon_{T^{*,1}_\varepsilon},T^{*,1}_\varepsilon)}
    {\|\varphi^\varepsilon(t;
    \hat{\varphi}^\varepsilon_{T^{*,1}_\varepsilon},T^{*,1}_\varepsilon)\|_{\omega}}\;\;\;\;
    \mbox{for all}\;\ t\in[0,T^{*,1}_\varepsilon).
\end{equation}
    Under this framework, an independent interesting result obtained in this study, which  plays an important role in the proof of Theorem~\ref{ithyu1.1}, is stated as follows:
    {\it
    The  minimizer of $J^{T^{*,1}_\varepsilon}_\varepsilon(\cdot)$ converges to the minimizer of $J^{T^*}(\cdot)$ strongly
    in $L^2(\Omega)$ as $\varepsilon$ tends to zero (see Theorem \ref{theoremyu1.1})}. This result, as well as (\ref{yu0.01}) and (\ref{yu0.02}), leads  to the $L^\infty$-convergence of the optimal controls
    stated in  Theorem~\ref{ithyu1.1}.
\par
    The second purpose of this paper is to construct a time optimal control problem for the perturbed equation (\ref{equationyu2.02}) such that this new problem has  the same optimal time as
    that of $(TP)$ and the optimal control for the new problem converges to that of  $ (TP)$. More precisely, we define a functional  $J^{T^*}_\varepsilon(\cdot)$ over $L^2(\Omega)$  by
\begin{eqnarray}\label{w-1}
    J^{T^*}_{\varepsilon}(\varphi^\varepsilon_{T^*})&=&\frac{1}{2}\left(\int_0^{T^*}
    \|\varphi^\varepsilon(t;\varphi^\varepsilon_{T^*},T^*)\|_\omega dt\right)^2
    +\langle y_0,\varphi^\varepsilon(0;\varphi_{T^*}^\varepsilon,T^*)\rangle_\Omega\nonumber\\
    &\;&+K\|\varphi^\varepsilon_{T^*}\|_\Omega,\;\varphi^\varepsilon_{T^*}\in L^2(\Omega),
\end{eqnarray}
    where $\varphi^\varepsilon(\cdot; \varphi_{T^*}^\varepsilon,T^*)$ is the solution of
\begin{equation}\label{w-2}
\begin{cases}
    \varphi^\varepsilon_t+\triangle\varphi^\varepsilon
    +a_\varepsilon\varphi^\varepsilon=0
    &\mbox{in}\;\;\Omega\times(0,T^*),\\
    \varphi^\varepsilon=0&\mbox{on}\;\;\partial\Omega\times(0,T^*),\\
    \varphi^\varepsilon(T^*)=\varphi^\varepsilon_{T^*}\in L^2(\Omega).
\end{cases}
\end{equation}
    The functional $J^{T^*}_\varepsilon(\cdot)$ has a unique minimizer  (see Section 4.2 in \cite {b3}), which is denoted by  $\hat{\varphi}_{T^*}^\varepsilon$.
    It is proved that  $\hat{\varphi}^\varepsilon_{T^*}\neq 0$, when $\varepsilon>0$ is small enough (see Step 1 in the proof of Proposition \ref{lemmayu4.1}).
    Let

\begin{equation}\label{yub2.01}
    M_\varepsilon=\int_0^{T^*}\|\varphi^\varepsilon(t; \hat{\varphi}_{T^*}^\varepsilon,T^*)\|_{\omega}dt,
\end{equation}
       and
       \vskip 5pt
     \noindent    $\mathcal{U}^{\varepsilon}_{M_\varepsilon}\equiv\{u\in
    L^\infty(\mathbb{R}^+;L^2(\Omega)):
    u(\cdot)\in
    \overline{B_{M_\varepsilon}(0)}\;\mbox{over}\;\mathbb{R}^+
    \;\mbox{and}\;
    \exists t>0\;\mbox{s.t.}\; y^{\varepsilon}(t;u,y_0)\in
    \overline{B_K(0)}\}.$

    \vskip 5pt

    \noindent Now we define the following time optimal control problem:

    \vskip 5pt

    \noindent
   $(TP^\varepsilon_2)\;\;\;\;
    T^{*,2}_\varepsilon\equiv\inf_{u\in\mathcal{U}^{\varepsilon}_{M_\varepsilon}}\{
    t\in\mathbb{R}^+:y^\varepsilon(t;u,y_0)\in\overline{B_K(0)}\}.$

    \vskip 5pt
    \noindent
    The second main result of this paper can be stated as follows:
\begin{theorem}\label{ithyu1.2}
   Suppose that  $(H_1)$, $(H_2)$ and $(H_3)$ hold.
    Let $T^*$, $u^*$ and $T^{*,2}_\varepsilon$, $u^{*,2}_\varepsilon$  be  the optimal time and the optimal controls  to
    Problems $(TP)$ and $(TP^\varepsilon_2)$ respectively. Then
   there exists an
    $\varepsilon_0>0$ such that

  \noindent   $(i)$ when  $\varepsilon\in(0,\varepsilon_0]$, $T^{*,2}_\varepsilon= T^*$;

  \noindent   $(ii)$ when  $\varepsilon\in(0,\varepsilon_0]$,  it holds that
\begin{equation}\label{yub4.10}
    u^{*,2}_\varepsilon(t)
    =M_\varepsilon
    \frac{\chi_\omega\varphi^\varepsilon(t;\hat{\varphi}^\varepsilon_{T^*},T^*)}
    {\|\varphi^\varepsilon(t;\hat{\varphi}^\varepsilon_{T^*},T^*)\|_{\omega}}\;\;\mbox{for each}\;\;
    t\in[0,T^*);
\end{equation}

  \noindent   $(iii)$ $M_\varepsilon\to M$ as $\varepsilon\to 0^+$;

   \noindent  $(iv)$ $u^{*,2}_\varepsilon\to u^*$ strongly in
    $L^2((0,T^*)\times\Omega)$ as $\varepsilon\to 0^+$;

  \noindent   $(v)$ for any $\eta\in(0,T^*)$, $u^{*,2}_\varepsilon\to u^*$
    strongly in $L^\infty(0,T^*-\eta;L^2(\Omega))$ as
    $\varepsilon\to 0^+$.

  \end{theorem}

\par
    The motivation for us to study such approximations presented in Theorem~\ref{ithyu1.1} and Theorem~\ref{ithyu1.2} are as follows.
    From the perspective  of applications,  the perturbations in the system potential often appears in  some physical phenomenons.
   It should be interesting and important to study how the perturbations influence some quantities related to the system without perturbations. The optimal control and the optimal time are such quantities. Our Theorem~\ref{ithyu1.1} and Theorem~\ref{ithyu1.2} reveal that the influence on the optimal control, as well as the optimal time, caused by small perturbations in the system potential, is small.  From the mathematical point of view,
   it deserves to mention that there have been a lot of papers studying  how the perturbations on the initial data influence the optimal time (see
   \cite{b15, b16, b20, b17, b19, b21, b18, b22} and references therein). However, to the best of our knowledge, such study when the perturbations appear in the system potential has not been touch upon.

    The rest of the paper is organized as follows:
   Section 2 studies  the associated minimization problems stated in the above and gives some preliminary results.  Section 3 introduces the equivalence theorem of the minimal time and norm control problems,  and provides some explicit formulas for the optimal
    controls to the problems studied in this paper. Section 4 presents  the proof of Theorem \ref{ithyu1.1} and Theorem \ref{ithyu1.2}. Some comments are given in Section 5.

\section{Preliminaries}
    In this section, we  present some preliminary  results about the time optimal control problems and the minimization problems associated with  the functionals $J^{T^*}(\cdot)$, $J^{T^*}_\varepsilon(\cdot)$ and $J^{T^{*,1}_\varepsilon}_\varepsilon(\cdot)$. We will write $S(\cdot)$  and $S^\varepsilon(\cdot)$ for the strongly continuous
    semigroups which are analytic and compact  generated by $(-\triangle-a)$ and $(-\triangle-a_\varepsilon)$ in $L^2(\Omega)$ respectively. (For the analyticity, we refer the reader to  Corollary 2.2 on page 81 in \cite{b25}.) By $(H_3)$, it holds that
    \begin{equation}\label{yu-b-1}
    \|S(t)\|\leq e^{-\delta_0 t}\;\;\mbox{for each}\;\;
    t\geq 0,
\end{equation}
    where
\begin{equation}
\delta_0\equiv
\begin{cases}
    \lambda_1-\|a\|_{L^\infty(\Omega)}
     &\mbox{if}\;\;\|a\|_{L^\infty(\Omega)}<\lambda_1,\\
    \lambda_1
     &\mbox{if}\;\;a(x)\leq 0\;\mbox{for any}\;x\in\Omega.
\end{cases}
\end{equation}
    By $(H_1)$ and $(H_3)$, there is an
    $\varepsilon_\rho>0$ such that
\begin{equation}\label{yu-b-2}
    \|S^\varepsilon(t)\|\leq e^{-\hat{\delta} t}\;\;\mbox{for each}\;\; t\in\mathbb{R}^+,\;\;\mbox{when}\;\; \varepsilon\in(0,\varepsilon_\rho],
\end{equation}
    where
     \begin{equation}\label{wang2.2}
    \hat{\delta}\equiv
\begin{cases}
    \frac{\lambda_1-\|a\|_{L^\infty(\Omega)}}{2}
     &\mbox{when}\;\;\|a\|_{L^\infty(\Omega)}<\lambda_1,\\
     \frac{\lambda_1}{2}
     &\mbox{when}\;\;a(x)\leq 0\;\mbox{for any}\;x\in\Omega.
\end{cases}
\end{equation}

      First of all, we introduce the following proposition concerning with the existence and the uniqueness of optimal
    controls for  Problems $(TP)$,
    $(TP^\varepsilon_1)$ and
    $(TP^\varepsilon_2)$.
\begin{proposition}\label{lemmayu3.1}
    Suppose that  $(H_1)$, $(H_2)$ and $(H_3)$ hold. Let $\varepsilon_\rho>0$ verify (\ref{yu-b-2}).      Then,  the problems $(TP)$, $(TP^\varepsilon_1)$ and $(TP^\varepsilon_2)$, with
     $\varepsilon\in (0, \varepsilon_\rho]$, have the unique optimal controls.
\end{proposition}

\begin{proof}
    By (\ref{yu-b-1}) and (\ref{yu-b-2}), we can use Lemma 3.2 in  \cite{b7} (or  Theorem 2 in \cite{b8}) to get the existence of optimal controls to the problems
    $(TP)$, $(TP^\varepsilon_1)$ and $(TP^\varepsilon_2)$, where
    $\varepsilon\in (0, \varepsilon_\rho]$.

      From  \cite{b10} (see Theorem 1  and Remark in and at the end of this paper), the problem $(TP)$ has the bang-bang property, i.e., any optimal controls $u^*(\cdot)$ to $(TP)$ satisfies that
      $\|u^*(\cdot)\|_\Omega=M$ for a.e. $t\in [0,T^*)$ with $T^*$ the optimal time to $(TP)$. The same can be said about the problems  $(TP_1^\varepsilon)$ and $(TP_2^\varepsilon)$.
           Then the uniqueness of the time optimal controls follows from the bang-bang property (see Theorem 2.1.7 on page 36 in \cite{b24} or Theorem 1.2 in \cite{b9}). This completes the proof.
\end{proof}
\par

    Next, we  introduce two minimization problems.
   Let $T>0$ and $T_\varepsilon>0$.  Define two functionals $J^T(\cdot)$ and $J^{T_\varepsilon}_\varepsilon(\cdot)$ over $L^2(\Omega)$ by
\begin{eqnarray}\label{101-1}
    J^T(\varphi_T)&=&\frac{1}{2}\left(\int_0^T\|\varphi(t;\varphi_T,T)\|_\omega
     dt\right)^2+\langle
    y_0,\varphi(0;\varphi_T,T)\rangle_\Omega\nonumber\\
    &\;&+K\|\varphi_T\|_\Omega,\;\;
    \varphi_T\in L^2(\Omega)
\end{eqnarray}
    and
\begin{eqnarray}\label{102-1}
    J^{T_\varepsilon}_\varepsilon(\varphi^\varepsilon_{T_\varepsilon})
    &=&\frac{1}{2}\left(\int_0^{T_\varepsilon}
    \|\varphi^\varepsilon(t;\varphi^\varepsilon_
    {T_\varepsilon},T_\varepsilon)\|_{\omega}dt\right)^2+\langle
    y_0,\varphi^\varepsilon(0;\varphi_{T_\varepsilon}
    ^\varepsilon,T_\varepsilon)\rangle_\Omega\nonumber\\
    &\;&+K\|\varphi^\varepsilon_{T_\varepsilon}\|_\Omega,\;\;
    \varphi^\varepsilon_{T_\varepsilon}\in L^2(\Omega),
\end{eqnarray}
     where $\varphi(\cdot;\varphi_T,T)$ and $\varphi^\varepsilon
    (\cdot;\varphi^\varepsilon_{T_\varepsilon},T_\varepsilon)$ are
    accordingly the solutions of the following equations
\begin{equation}\label{103-1}
\begin{cases}
    \varphi_t+\triangle\varphi+a\varphi=0&\mbox{in}\;\;\Omega\times(0,T),\\
    \varphi=0&\mbox{on}\;\;\partial\Omega\times(0,T),\\
    \varphi(T)=\varphi_T\in L^2(\Omega)
\end{cases}
\end{equation}
    and
\begin{equation}\label{104-1}
\begin{cases}
    \varphi^\varepsilon_t+\triangle\varphi^\varepsilon
    +a_\varepsilon \varphi^\varepsilon=0
    &\mbox{in}\;\;\Omega\times(0,T_\varepsilon),\\
    \varphi^\varepsilon=0&\mbox{on}\;\;
    \partial\Omega\times(0,T_\varepsilon),\\
    \varphi^\varepsilon(T_\varepsilon)
    =\varphi^\varepsilon_{T_\varepsilon}\in L^2(\Omega).
\end{cases}
\end{equation}
   Consider  two minimization
    problems as follows:

    \vskip 5pt

  \noindent $(MP)\;\;$ To find  $\hat{\varphi}_T$
    in $L^2(\Omega)$ such that
\begin{equation}\label{105-1}
    J^T(\hat{\varphi}_T)=\inf_{\varphi_T\in L^2(\Omega)}J^T(\varphi_T);
\end{equation}

    \noindent $(MP^\varepsilon)\;\; $ To find  $\hat{\varphi}^\varepsilon_{T_\varepsilon}$ in $L^2(\Omega)$ such that
\begin{equation}\label{106-1}
    J^{T_\varepsilon}_\varepsilon(\hat{\varphi}^\varepsilon_{T_\varepsilon})
    =\inf_{\varphi^\varepsilon_{T_\varepsilon}\in
    L^2(\Omega)}J^{T_\varepsilon}_\varepsilon(\varphi^\varepsilon_{T_\varepsilon}).
\end{equation}
\noindent
    We assume that
    \vskip 5pt

    \noindent $(H_4)$ $|T_\varepsilon-T|\to 0$ as $\varepsilon\to 0^+$.
    \vskip 5 pt

    \noindent $(H_5)$ $y_0\in L^2(\Omega)$ such
    that $y(T; y_0)\notin \overline{B_K(0)}$, where
    $y(T; y_0)\equiv y(T;0,y_0)$.

\vskip 5pt
\par
\noindent     Clearly, when $T_\varepsilon\equiv T$ for all $\varepsilon>0$, it holds that
\begin{equation}\label{b11}
    \|\varphi^\varepsilon(\cdot;\varphi_T,T)-\varphi(\cdot;\varphi_T,T)
    \|_{C([0,T];L^2(\Omega))}\to
    0~~\text{as}~\varepsilon\to 0^+,\;\;\mbox{for any}\;\; \varphi_T\in L^2(\Omega).
\end{equation}

   \begin{lemma}\label{lemmawang1}
   $(i)$ The functional $J^T(\cdot)$ has a unique minimizer.
   $(ii)$ The minimizer of $J^T(\cdot)$
   is not zero if and only if $(H_5)$ holds. $(iii)$ For each
   $\varepsilon>0$, the functional
   $J^{T_\varepsilon}_\varepsilon(\cdot)$ has a unique minimizer.
   $(iv)$ Suppose that $(H_1)$, $(H_4)$ and $(H_5)$ hold.
   Then there is an $\varepsilon_0>0$ such that
    each $J^{T_\varepsilon}_\varepsilon(\cdot)$, with $\varepsilon\in (0, \varepsilon_0]$,
    has a unique non-zero minimizer.
   \end{lemma}
   \begin{proof}
   By a  very similar argument as Proposition 2.1 in \cite{b6} (see also Section 4.2 in \cite{b3}),
    one can easily check that both $J^T(\cdot)$
    and $J^{T_\varepsilon}_\varepsilon(\cdot)$, with  $\varepsilon>0$,
    are continuous, coercive and strictly convex. Hence, they have
    unique minimizers. Moreover,
    the minimizer of $J(\cdot)$ is not zero if and only if $(H_5)$ holds, i.e.,
    $y(T; y_0)\notin \overline{B_K(0)}$. The rest is to show $(iv)$. From $(H_4)$,
    $|T_\varepsilon-T|\to 0$ as $\varepsilon\to 0^+$. Given $\delta\in(0,T)$,
    there exists an $\varepsilon_1(\delta)>0$ such that
    $|T_\varepsilon-T|\leq \delta$, when $\varepsilon\in
    (0,\varepsilon_1(\delta)]$. Write
    $y(\cdot;y_0)$ and $y^\varepsilon(\cdot;y_0)$ for the solutions to the equations
\begin{equation}\label{iequationyu2.1}
\begin{cases}
    y_t-\triangle y-a
    y=0&\mbox{in}\;\;\Omega\times(0,T+\delta),\\
    y=0&\mbox{on}\;\;\partial\Omega\times(0,T+\delta),\\
    y(0)=y_0&\mbox{in}\;\;\Omega,
\end{cases}
\end{equation}
\begin{equation}\label{equationwang2}
\begin{cases}
    y^\varepsilon_t-\triangle y^\varepsilon-a_\varepsilon
    y^\varepsilon=0&\mbox{in}\;\;\Omega\times(0,T+\delta),\\
    y^\varepsilon=0&\mbox{on}\;\;\partial\Omega\times(0,T+\delta),\\
    y^\varepsilon(0)=y_0&\mbox{in}\;\;\Omega.
\end{cases}
\end{equation}
       From the equations (\ref{iequationyu2.1}) and (\ref{equationwang2}), and
    by the assumption  $(H_1)$,
    one can easily derive that
 \begin{equation}\label{b14}
    \|y^\varepsilon(\cdot;y_0)-y(\cdot;y_0)\|_{C([0,T+\delta];L^2(\Omega))}\to
    0~~\text{as}~\varepsilon\to 0^+.
\end{equation}
    This implies
\begin{eqnarray*}
    \|y^\varepsilon(T_\varepsilon;y_0)-y(T;y_0)\|_\Omega
    &\leq&\|y^\varepsilon(\cdot;y_0)-y(\cdot;y_0)\|_{C([0,T+\delta];L^2(\Omega))}\\
    &\;&+\|y(T_\varepsilon;y_0)-y(T;y_0)\|_\Omega\to
    0~~\text{as}~\varepsilon\to 0^+.
\end{eqnarray*}
    Because of $(H_5)$, there is an $\varepsilon_2(\delta)\leq
    \varepsilon_1(\delta)$ such that
\begin{equation}\label{127}
    y^\varepsilon(T_\varepsilon;y_0)\notin
    \overline{B_K(0)}~~\mbox{for each}\;\; \varepsilon\in (0, \varepsilon_2(\delta)].
\end{equation}
    Now, by taking $\varepsilon_0=\varepsilon_2(\delta)$ and making use
    of the conclusion $(ii)$, we are led to $(iv)$. This completes the proof.
   \end{proof}
\par
   In what follows,  we fix a $\delta\in(0,T)$ and let
   $\varepsilon_0=\varepsilon_0(\delta)>0$ be such that
  \begin{equation}\label{wang1.13}
  |T-T_\varepsilon|\leq \delta\;\;\mbox{and}\;\; arg \min
  J^{T^\varepsilon}_\varepsilon(\cdot)\neq 0,\;\;\mbox{when}\;\; \varepsilon\in (0, \varepsilon_0].
  \end{equation}

\begin{theorem}\label{theoremyu1.1}
    Suppose that $(H_1)$, $(H_4)$ and $(H_5)$ hold. Let
    $\varepsilon_0$ be given by (\ref{wang1.13}).
    Let $\hat{\varphi}_T$ and $\hat{\varphi}_{T_\varepsilon}^\varepsilon$, with $\varepsilon\in (0, \varepsilon_0]$,
     be accordingly the non-zero
    minimizers of functionals  $J^T(\cdot)$ and $J^{T_\varepsilon}_\varepsilon(\cdot)$ defined by (\ref{101-1}) and (\ref{102-1}). Then
    $\hat{\varphi}_{T_\varepsilon}^\varepsilon\to \hat{\varphi}_T$ strongly in
    $L^2(\Omega)$ as $\varepsilon\to 0^+$.
\end{theorem}
\begin{proof}
    We start with showing  the existence of such  $C>0$, independent of $\varepsilon$,  that
\begin{equation}\label{wang1.14}
    \|\hat{\varphi}_{T_\varepsilon}^\varepsilon\|_\Omega\leq
    C\;\;\mbox{for all}\; \; \varepsilon\in (0, \varepsilon_0].
\end{equation}
    For this purpose, we first observe that
\[
    J^{T_\varepsilon}_\varepsilon(\hat{\varphi}_{T_\varepsilon}^\varepsilon)\leq
    0\;\;\mbox{for each}\;\;  \varepsilon\in (0, \varepsilon_0].
\]
    Along with (\ref{102-1}), this yields
\begin{equation}\label{129}
    K\|\hat{\varphi}_{T_\varepsilon}^\varepsilon\|_\Omega\leq
    -\frac{1}{2}\left(\int_0^{T_\varepsilon}\|\varphi^\varepsilon
    (t;\hat{\varphi}^\varepsilon_{T_\varepsilon},T_\varepsilon)\|_{\omega}dt\right)^2
    +|\langle
    y_0,\varphi^\varepsilon(0;\hat{\varphi}^\varepsilon_{T_\varepsilon},T_\varepsilon)
    \rangle_\Omega|.
\end{equation}
    From Proposition 3.2 in \cite{b1}, we have
\begin{eqnarray}\label{117}
    &\;&\|\varphi^\varepsilon(0;\hat{\varphi}^\varepsilon_{T_\varepsilon},
    T_\varepsilon)\|_\Omega\leq
    \exp\biggl[C_0\bigg(1+\frac{1}{T_\varepsilon}+T_\varepsilon+
    (T^{\frac{1}{2}}_\varepsilon+T_\varepsilon)
    \|a_\varepsilon\|_{L^\infty(\Omega)}\nonumber\\
    &\;&~~~~~~~+\|a_\varepsilon\|_{L^\infty(\Omega)}^{\frac{2}{3}}\biggl)\biggl]
    \int_0^{T_\varepsilon}\|\varphi^\varepsilon
    (t;\hat{\varphi}^\varepsilon_{T_\varepsilon},T_\varepsilon)
    \|_{\omega}dt,
\end{eqnarray}
    where  $C_0$ is a positive constant  depending only on
    $\Omega$ and $\omega$.  From this, $(H_1)$ and $(H_4)$,
    there exists a constant $C_1>0$ such that
$$
    \exp\left[C_0\biggl(1+\frac{1}{T_\varepsilon}+T_\varepsilon
    +(T_\varepsilon^{\frac{1}{2}}+
    T_\varepsilon)\|
    a_\varepsilon\|_{L^\infty(\Omega)}
    +\|a_\varepsilon\|_{L^\infty(\Omega)}^{\frac{2}{3}}\biggl)\right]
    \leq C_1,\;\;\mbox{when}\;\; \varepsilon\in (0, \varepsilon_0].
$$
    This, together with  (\ref{129}) and (\ref{117}), leads to (\ref{wang1.14}).
\par
    Next, we arbitrarily take a sequence  $\{\varepsilon_n\}_{n\in\mathbb{N}}
    \subset \{\varepsilon\}_{\varepsilon\in (0, \varepsilon_0]}$ such that
    $\varepsilon_n\to 0^+$ as $n\to \infty$.
    By (\ref{wang1.14}), there exists a subsequence of the above
    sequence, still denoted in the same way,
    such that
\begin{equation}\label{yu1.18}
    \hat{\varphi}_{T_{\varepsilon_n}}^{\varepsilon_n}\to
    \tilde{\varphi}~~\text{weakly in}~L^2(\Omega)~\text{as}~n\to
    \infty,
\end{equation}
    where $\tilde{\varphi}\in L^2(\Omega)$.
    We are going to prove
\begin{equation}\label{wang1.18}
    \tilde{\varphi}=\hat{\varphi}_T.
\end{equation}
    When (\ref{wang1.18}) is proved, by the arbitrariness of $\{\varepsilon_n\}_{n\in\mathbb{N}}\subset\{\varepsilon\}
    _{\varepsilon\in(0,\varepsilon_0]}$, we are led to
\begin{equation}\label{yu1.19}
    \hat{\varphi}_{T_\varepsilon}^\varepsilon\to
    \hat{\varphi}_T~~\text{weakly
    in}~L^2(\Omega)~\text{as}~\varepsilon\to 0^+.
\end{equation}
\par
    To show (\ref{wang1.18}),
    we first prove two statements as follows:
\begin{equation}\label{136}
    |\langle
    y_0,\varphi^{\varepsilon_n}(0;\hat{\varphi}_{T_{\varepsilon_n}}^{\varepsilon_n},
    T_{\varepsilon_n})
    -\varphi(0;\hat{\varphi}_{T_{\varepsilon_n}}^{\varepsilon_n},T)\rangle_\Omega|\to
    0~~\text{as}~n\to \infty
\end{equation}
    and
\begin{equation}\label{137}
   \left|\left(\int_0^{T_{\varepsilon_n}}\|\varphi^{\varepsilon_n}
    (t;\hat{\varphi}_{T_{\varepsilon_n}}^{\varepsilon_n},T_{\varepsilon_n})\|_{\omega}dt\right)^2
    -\left(\int_0^T\|\varphi(t;\hat{\varphi}_{T_{\varepsilon_n}}
    ^{\varepsilon_n},T)\|_{\omega}dt\right)^2\right|
    \to 0~~\text{as}~n\to \infty.
\end{equation}
\par
    The proof of  (\ref{136}) is as follows:
    We recall that  $\delta>0$ was fixed by (\ref{wang1.13}). For each $\psi_T\in L^2(\Omega)$, we denote by
    $\psi^{\varepsilon_n}(\cdot;\psi_T,0)$ and  $\psi(\cdot;\psi_T,0)$
     the solutions to the following two equations respectively:
\begin{equation}
\begin{cases}
    \psi_t^{\varepsilon_n}-\triangle\psi^{\varepsilon_n}
    -a_{\varepsilon_n}\psi^{\varepsilon_n}=0&\mbox{in}\;\;\Omega\times(0,T+\delta),\\
    \psi^{\varepsilon_n}=0&\mbox{on}\;\;\partial\Omega\times(0,T+\delta),\\
    \psi^{\varepsilon_n}(0)=\psi_T&\mbox{in}\;\;\Omega,
\end{cases}
\end{equation}
    and
\begin{equation}\label{132}
\begin{cases}
    \psi_t-\triangle\psi-a\psi=0&\text{in}\;\;\Omega\times(0,T+\delta),\\
    \psi=0&\text{on}\;\;\partial\Omega\times(0,T+\delta),\\
    \psi(0)=\psi_T&\text{in}\;\;\Omega.
\end{cases}
\end{equation}
    It is clear that
\begin{equation}\label{wang1.21}
    \varphi^{\varepsilon_n}(t;\hat{\varphi}_{T_{\varepsilon_n}}
    ^{\varepsilon_n},T_{\varepsilon_n})
    =\psi^{\varepsilon_n}(T_{\varepsilon_n}-t;\hat{\varphi}
    _{T_{\varepsilon_n}}^{\varepsilon_n},0)\;\;\mbox{for all}\;\;
    t \in [0,T_{\varepsilon_n}];
\end{equation}
\begin{equation}\label{wang1.23}
    \varphi(t;\hat{\varphi}_{T_{\varepsilon_n}}^{\varepsilon_n},T)=\psi(T-t;
    \hat{\varphi}_{T_{\varepsilon_n}}^{\varepsilon_n},0)\;\;\mbox{for all}
    \;\; t \in[0,T].
\end{equation}
    From (\ref{wang1.14}), there is a $C>0$,  independent of $n$,  such that
\begin{equation}\label{wang1.24}
    \|\psi^{\varepsilon_n}(\cdot;\hat{\varphi}_{T_{\varepsilon_n}}
    ^{\varepsilon_n},0)\|_{C([0,T+\delta];L^2(\Omega))}
    +\|\psi(\cdot;\hat{\varphi}_{T_{\varepsilon_n}}^{\varepsilon_n},0)
    \|_{C([0,T+\delta];L^2(\Omega))}\leq
    C\;\;\mbox{for all}\;\; n\in\mathbb{N}
\end{equation}
    and
\begin{equation}\label{yu1.27}
    \|\psi^{\varepsilon_n}(\cdot;\hat{\varphi}_{T_{\varepsilon_n}}^{\varepsilon_n},0)
    -\psi(\cdot;\hat{\varphi}_{T_{\varepsilon_n}}^{\varepsilon_n},0)
    \|_{C([0,T+\delta];L^2(\Omega))}\to
    0\;\;\text{as}\;\; n\to \infty.
\end{equation}
   From (\ref{132}),
    the strong continuity and compactness of $S(\cdot)$ and the fact that $\delta\in(0,T)$, it follows that there exists a subsequence of $\{\varepsilon_n\}$, still denoted by the same way, such that
\begin{eqnarray}\label{yu1.281}
    &\;&\|\psi(T_{\varepsilon_n};\hat{\varphi}_{T_{\varepsilon_n}}^{\varepsilon_n},0)
    -\psi(T;\hat{\varphi}_{T_{\varepsilon_n}}^{\varepsilon_n},0)\|_\Omega\nonumber\\
    &=&\|S(T_{\varepsilon_n})\hat{\varphi}_{T_{\varepsilon_n}}^{\varepsilon_n}
    -S(T)\hat{\varphi}_{T_{\varepsilon_n}}^{\varepsilon_n}\|_\Omega\nonumber\\
    &\leq&\|S(T_{\varepsilon_n})\hat{\varphi}_{T_{\varepsilon_n}}^{\varepsilon_n}
    -S(T_{\varepsilon_n})\tilde{\varphi}\|_\Omega
    +\|S(T_{\varepsilon_n})\tilde{\varphi}-S(T)\tilde{\varphi}\|_\Omega\nonumber\\
    &\;&+\|S(T)\hat{\varphi}_{T_{\varepsilon_n}}^{\varepsilon_n}
    -S(T)\tilde{\varphi}\|_\Omega\nonumber\\
    &=&\|S(T_{\varepsilon_n}-T+\delta)
    [S(T-\delta)\hat{\varphi}_{T_{\varepsilon_n}}^{\varepsilon_n}
    -S(T-\delta)\tilde{\varphi}]\|_\Omega\nonumber\\
    &\;&+\|S(T_{\varepsilon_n})\tilde{\varphi}-S(T)\tilde{\varphi}\|_\Omega
    +\|S(T)\hat{\varphi}_{T_{\varepsilon_n}}^{\varepsilon_n}
    -S(T)\tilde{\varphi}\|_\Omega\nonumber\\
    &\leq&\|S(T-\delta)\hat{\varphi}_{T_{\varepsilon_n}}^{\varepsilon_n}
    -S(T-\delta)\tilde{\varphi}\|_\Omega+\|S(T_{\varepsilon_n})
    \tilde{\varphi}-S(T)\tilde{\varphi}\|_\Omega\nonumber\\
    &\;&
    +\|S(T)\hat{\varphi}_{T_{\varepsilon_n}}^{\varepsilon_n}
    -S(T)\tilde{\varphi}\|_\Omega
    \to 0\;\;\text{as}\;\; n\to 0.
\end{eqnarray}
    This, together with (\ref{wang1.21}) and (\ref{wang1.23}), indicates
\begin{eqnarray*}
    &\;&\|\varphi^{\varepsilon_n}(0;\hat{\varphi}_{T_{\varepsilon_n}}
    ^{\varepsilon_n},T_{\varepsilon_n})
    -\varphi(0;\hat{\varphi}_{T_{\varepsilon_n}}^{\varepsilon_n},T)\|_\Omega
    =\|\psi^{\varepsilon_n}(T_{\varepsilon_n};
    \hat{\varphi}_{T_{\varepsilon_n}}^{\varepsilon_n},0)
    -\psi(T;\hat{\varphi}_{T_{\varepsilon_n}}^{\varepsilon_n},0)\|_\Omega\\
    &\leq&\|\psi^{\varepsilon_n}(T_{\varepsilon_n};
    \hat{\varphi}_{T_{\varepsilon_n}}^{\varepsilon_n},0)
    -\psi(T_{\varepsilon_n};\hat{\varphi}_{T_{\varepsilon_n}}
    ^{\varepsilon_n},0)\|_\Omega
    +\|\psi(T_{\varepsilon_n};\hat{\varphi}_{T_{\varepsilon_n}}^{\varepsilon_n},0)
    -\psi(T;\hat{\varphi}_{T_{\varepsilon_n}}^{\varepsilon_n},0)\|_\Omega\\
    &\leq&\|\psi^{\varepsilon_n}(\cdot;\hat{\varphi}_{T_{\varepsilon_n}}
    ^{\varepsilon_n},0)
    -\psi(\cdot;\hat{\varphi}_{T_{\varepsilon_n}}^{\varepsilon_n},0)
    \|_{C([0,T+\delta];L^2(\Omega))}
    +\|\psi(T_{\varepsilon_n};\hat{\varphi}_{T_{\varepsilon_n}}^{\varepsilon_n},0)
    -\psi(T;\hat{\varphi}_{T_{\varepsilon_n}}^{\varepsilon_n},0)\|_\Omega\\
    &\;& \to 0\;\;\mbox{as}\;\; n\to\infty.
\end{eqnarray*}
    Hence,
\begin{eqnarray}\label{yu1.28}
    &\;&|\langle
    y_0,\varphi^{\varepsilon_n}(0;\hat{\varphi}_{T_{\varepsilon_n}}
    ^{\varepsilon_n},T_{\varepsilon_n})
    -\varphi(0;\hat{\varphi}_{T_{\varepsilon_n}}^{\varepsilon_n},T)
    \rangle_\Omega|\nonumber\\
    &\leq&\|y_0\|_{\Omega}\|\varphi^{\varepsilon_n}
    (0;\hat{\varphi}_{T_{\varepsilon_n}}^{\varepsilon_n},T_{\varepsilon_n})
    -\varphi(0;\hat{\varphi}_{T_{\varepsilon_n}}^{\varepsilon_n},T)
    \|_\Omega\to 0\;\;\mbox{as}\;\;n\to \infty,
\end{eqnarray}
    which leads to  (\ref{136}).
\par
    The proof of  (\ref{137}) is as follows:
     By (\ref{yu1.18})  and by using Aubin's theorem (see Theorem 1.20 on page 26 in \cite{b2}), for any  subsequence $\{\varepsilon_{n_i}\}_{i\in\mathbb{N}}\subset\{\varepsilon_n\}_{n\in\mathbb{N}}$, there exists a subsequence of $\{\varepsilon_{n_i}\}$, still denoted in the same way, such that
\begin{equation}\nonumber
    \|\psi(\cdot;\hat{\varphi}_{T_{\varepsilon_{n_i}}}^{\varepsilon_{n_i}},0)
    -\psi(\cdot;\tilde{\varphi},0)\|_{L^2(0,T+\delta;L^2(\Omega))}\to
    0~~\text{as}~~i\to \infty.
\end{equation}
    Since $\{\varepsilon_{n_i}\}_{i\in \mathbb{N}}$ was arbitrarily taken from $\{\varepsilon_n\}_{n\in\mathbb{N}}$, we have
\begin{equation}\label{yu133}
    \|\psi(\cdot;\hat{\varphi}_{T_{\varepsilon_n}}^{\varepsilon_n},0)
    -\psi(\cdot;\tilde{\varphi},0)\|_{L^2(0,T+\delta;L^2(\Omega))}\to
    0~~\text{as}~~n\to \infty.
\end{equation}
    It follows from (\ref{yu1.27}) that
\begin{equation}\label{134}
    \|\psi^{\varepsilon_n}(\cdot;\hat{\varphi}_{T_{\varepsilon_n}}^{\varepsilon_n},0)
    -\psi(\cdot;\tilde{\varphi},0)\|_{L^2(0,T+\delta;L^2(\Omega))}\to
    0~~\text{as}~~n\to \infty.
\end{equation}
   Meanwhile,  one can easily check that
\begin{eqnarray}\label{yu1.30}
    &\;&\left|\left(\int_0^{T_{\varepsilon_n}}\|\varphi^{\varepsilon_n}
    (t;\hat{\varphi}_{T_{\varepsilon_n}}^{\varepsilon_n},
    T_{\varepsilon_n})\|_{\omega}dt\right)^2
    -\left(\int_0^T\|\varphi(t;\hat{\varphi}_{T_{\varepsilon_n}}^{\varepsilon_n},T)
    \|_{\omega}dt\right)^2\right|\nonumber\\
    &\leq&\left|\int_0^{T_{\varepsilon_n}}\|\varphi^{\varepsilon_n}
    (t;\hat{\varphi}_{T_{\varepsilon_n}}^{\varepsilon_n},T_{\varepsilon_n})\|_{\omega}dt
    +\int_0^T\|\varphi(t;\hat{\varphi}_{T_{\varepsilon_n}}^{\varepsilon_n},T)
    \|_{\omega}dt\right|\nonumber\\
    &\;&~~~~~\times\left|\int_0^{T_{\varepsilon_n}}\|\varphi^{\varepsilon_n}
    (t;\hat{\varphi}_{T_{\varepsilon_n}}^{\varepsilon_n},T_{\varepsilon_n})\|_{\omega}dt
    -\int_0^T\|\varphi(t;\hat{\varphi}_{T_{\varepsilon_n}}^{\varepsilon_n},T)
    \|_{\omega}dt\right|\nonumber\\
    &=&\left|\int_0^{T_{\varepsilon_n}}\|\psi^{\varepsilon_n}
    (T_{\varepsilon_n}-t;\hat{\varphi}_{T_{\varepsilon_n}}
    ^{\varepsilon_n},0)\|_{\omega}dt
    +\int_0^T\|\psi(T-t;\hat{\varphi}_{T_{\varepsilon_n}}^{\varepsilon_n},0)
    \|_{\omega}dt\right|\nonumber\\
    &\;&~~~~~\times\left|\int_0^{T_{\varepsilon_n}}\|\psi^{\varepsilon_n}
    (T_{\varepsilon_n}
    -t;\hat{\varphi}_{T_{\varepsilon_n}}^{\varepsilon_n},0)\|_{\omega}dt
    -\int_0^T\|\psi(T-t;\hat{\varphi}_{T_{\varepsilon_n}}^{\varepsilon_n},0)
    \|_{\omega}dt\right|\nonumber\\
    &\equiv& A_n\times B_n.
\end{eqnarray}
    By (\ref{wang1.24}), we have
\begin{eqnarray}\label{yu1.31}
    A_n\leq 2C(T+\delta).
\end{eqnarray}
   Now we are going to prove that $B_n\to 0$ as $n\to\infty$.
   Indeed,
\begin{eqnarray}\label{yu1.32}
    B_n&\leq&\int_0^{T_{\varepsilon_n}\wedge
    T}\left|\|\psi^{\varepsilon_n}(T_{\varepsilon_n}-t;
    \hat{\varphi}_{T_{\varepsilon_n}}^{\varepsilon_n},0)\|_{\omega}-
    \|\psi(T-t;\hat{\varphi}_{T_{\varepsilon_n}}^{\varepsilon_n},0)
    \|_{\omega}\right|dt\nonumber\\
    &\;&+\int_{T_{\varepsilon_n}\wedge T}^{T_{\varepsilon_n}\vee T}
    \left|\|\tilde{\psi}^{\varepsilon_n}
    (T_{\varepsilon_n}-t;\hat{\varphi}_{T_{\varepsilon_n}}^{\varepsilon_n},0)\|_{\omega}
    -\|\tilde{\psi}(T-t;\hat{\varphi}_{T_{\varepsilon_n}}^{\varepsilon_n},0)
    \|_{\omega}\right|dt\nonumber\\
    &\equiv& E_n+F_n,
\end{eqnarray}
    where $T_1\vee T_2=\max(T_1,T_2)$
    and $T_1\wedge T_2=\min(T_1,T_2)$ for any $T_1, T_2\in\mathbb{R}$,
    and
$$
    \tilde{\psi}^{\varepsilon_n}(t;\hat{\varphi}_{T_{\varepsilon_n}}^{\varepsilon_n},0)=
\begin{cases}
    \psi^{\varepsilon_n}(t;\hat{\varphi}_{T_{\varepsilon_n}}
    ^{\varepsilon_n},0),&t\geq 0,\\
    ~~~~~0,&t<0,
\end{cases}
$$
$$
    \tilde{\psi}(t;\hat{\varphi}_{T_{\varepsilon_n}}^{\varepsilon_n},0)=
\begin{cases}
    \psi(t;\hat{\varphi}_{T_{\varepsilon_n}}^{\varepsilon_n},0),
    &t\geq 0,\\
    ~~~~~0,&t<0.
\end{cases}
$$
    It is clear
\begin{eqnarray}\label{yu1.33}
    E_n&\leq&\int_0^{T_{\varepsilon_n}\wedge
    T}\|\psi^{\varepsilon_n}(T_{\varepsilon_n}-t;
    \hat{\varphi}_{T_{\varepsilon_n}}^{\varepsilon_n},0)
    -\psi(T_{\varepsilon_n}-t;
    \hat{\varphi}_{T_{\varepsilon_n}}^{\varepsilon_n},0)\|_{\omega}dt\nonumber\\
    &\;&+\int_0^{T_{\varepsilon_n}\wedge
    T}\|\psi(T_{\varepsilon_n}-t;
    \hat{\varphi}_{T_{\varepsilon_n}}^{\varepsilon_n},0)
    -\psi(T-t;\hat{\varphi}_{T_{\varepsilon_n}}^{\varepsilon_n},0)
    \|_{\omega}dt\nonumber\\
    &\equiv& E^1_n+E^2_n.
\end{eqnarray}
    By (\ref{yu1.27}), we see that, when $n\to\infty$,
\begin{eqnarray}\label{yu1.34}
    E^1_n&\leq&
    \int_0^{T+\delta}\|\psi^{\varepsilon_n}
    (t;\hat{\varphi}_{T_{\varepsilon_n}}^{\varepsilon_n},0)
    -\psi(t;\hat{\varphi}_{T_{\varepsilon_n}}^{\varepsilon_n},0)\|_\Omega dt\nonumber\\
    &\leq& (T+\delta)\|\psi^{\varepsilon_n}
    (\cdot;\hat{\varphi}_{T_{\varepsilon_n}}^{\varepsilon_n},0)
    -\psi(\cdot;\hat{\varphi}_{T_{\varepsilon_n}}^{\varepsilon_n},0)
    \|_{C([0,T+\delta];L^2(\Omega))} \to 0.
\end{eqnarray}
   Let
    $z^{\varepsilon_n}(t)=\psi(T_{\varepsilon_n}
    -t;\hat{\varphi}_{T_{\varepsilon_n}}^{\varepsilon_n},0)
    -\psi(T-t;\hat{\varphi}_{T_{\varepsilon_n}}^{\varepsilon_n},0)$. Then   for any $t\in[0,T_{\varepsilon_n}\wedge T]$, it holds that
\begin{eqnarray}\label{yu1.35}
    \|z^{\varepsilon_n}(t)\|_\Omega&=&\|S(T_{\varepsilon_n}-t)
    \hat{\varphi}^{\varepsilon_n}_{T_{\varepsilon_n}}
    -S(T-t)\hat{\varphi}^{\varepsilon_n}_{T_{\varepsilon_n}}\|_\Omega\nonumber\\
    &\leq&\|S(T_{\varepsilon_n}-t)\hat{\varphi}^{\varepsilon_n}_{T_{\varepsilon_n}}
    -S(T_{\varepsilon_n}-t)\tilde{\varphi}\|_\Omega\nonumber\\
    &\;&+\|S(T_{\varepsilon_n}-t)\tilde{\varphi}
    -S(T-t)\tilde{\varphi}\|_\Omega\nonumber\\
    &\;&+\|S(T-t)\tilde{\varphi}-S(T-t)
    \hat{\varphi}^{\varepsilon_n}_{T_{\varepsilon_n}}\|_\Omega.
\end{eqnarray}
   By the strong continuity of
    $S(\cdot)$ over $[0,
    T+\delta]$, (\ref{yu133}) and (\ref{yu1.35}), we have
\begin{eqnarray}\label{yu1.36}
    E_n^2 &=&\int_0^{T_{\varepsilon_n}\wedge
    T}\|\psi(T_{\varepsilon_n}-t;
    \hat{\varphi}_{T_{\varepsilon_n}}^{\varepsilon_n},0)
    -\psi(T-t;\hat{\varphi}_{T_{\varepsilon_n}}^{\varepsilon_n},0)
    \|_{\omega}dt\nonumber\\
    &\leq&2T^{\frac{1}{2}}\|
    \psi(\cdot;\hat{\varphi}_T^{\varepsilon_n},0)
    -\psi(\cdot;\tilde{\varphi},0)\|_{L^2(0,T+\delta;L^2(\Omega))}\nonumber\\
    &\;&+\int_0^{T_{\varepsilon_n}\wedge
    T}\|S(T_{\varepsilon_n}-t)\tilde{\varphi}-S(T-t)\tilde{\varphi}\|_\Omega dt
    \to 0
    \;\;\mbox{as}\;\; n\to\infty.
\end{eqnarray}
    This, together with (\ref{yu1.34}), yields
\begin{equation}\label{yu1.37}
    E_n\to 0\;\;\mbox{as}\;\;n\to\infty.
\end{equation}
  On the other hand, one can easily derive from (\ref{wang1.24}) that
$$
    F_n\leq 2C(T_{\varepsilon_n}\vee
    T-T_{\varepsilon_n}\wedge T)\to 0\;\;\text{as}\;\;n\to\infty.
$$
  Along with (\ref{yu1.37}), this yields
$$
    B_n\to 0\;\;\mbox{as}\;\;n\to\infty.
$$
  This, along with   (\ref{yu1.30}) and (\ref{yu1.31}), leads to  (\ref{137}).
\par
    Let
\begin{eqnarray*}
    I_n&\equiv&|\langle
    y_0,\varphi^{\varepsilon_n}(0;\hat{\varphi}_{T_{\varepsilon_n}}^{\varepsilon_n},
    T_{\varepsilon_n})
    -\varphi(0;\hat{\varphi}_{T_{\varepsilon_n}}^{\varepsilon_n},T)\rangle_\Omega|\\
    &\;&+\frac{1}{2}\left|\left(\int_0^{T_{\varepsilon_n}}\|\varphi^{\varepsilon_n}
    (t;\hat{\varphi}_{T_{\varepsilon_n}}^{\varepsilon_n},
    T_{\varepsilon_n})\|_{\omega}dt\right)^2
    -\left(\int_0^T\|\varphi(t;\hat{\varphi}_{T_{\varepsilon_n}}
    ^{\varepsilon_n},T)\|_{\omega}dt\right)^2\right|.
\end{eqnarray*}
    By (\ref{136}) and (\ref{137}), we see that
\begin{equation}\label{yu1.41}
    I_n\rightarrow 0\;\;\mbox{as}\;\;n\to\infty.
\end{equation}
   By the similar methods as those used in the proof of
    (\ref{136}) and (\ref{137}), one can easily check that
\begin{eqnarray}\label{yu-bu-29}
    &\;&|J^{T_{\varepsilon_n}}_{\varepsilon_n}(\hat{\varphi}_T)
    -J^T(\hat{\varphi}_T)|\leq
    |\langle
    y_0,\varphi^{\varepsilon_n}(0;\hat{\varphi}_T,T_{\varepsilon_n})
    -\varphi(0;\hat{\varphi}_T,T)\rangle_\Omega|\nonumber\\
    \;&\;&~~~~+\frac{1}{2}\left|\left(\int_0^{T_{\varepsilon_n}}\|\varphi^{\varepsilon_n}
    (t;\hat{\varphi}_T,T_{\varepsilon_n})\|_{\omega}dt\right)^2
    -\left(\int_0^T\|\varphi(t;\hat{\varphi}_T,T)\|_{\omega}dt\right)^2\right|\nonumber\\
    &\;&~~~~\to 0~~\text{as}~n\to\infty.
\end{eqnarray}
    Meanwhile, from (\ref{yu133}) and the compactness of $S(\cdot)$, we have
\begin{eqnarray*}
    &\;&\left|\int_0^T\|\varphi(t;\hat{\varphi}_{T_{\varepsilon_n}}
    ^{\varepsilon_n},T)\|_\omega-\int_0^T\|\varphi(t;\tilde{\varphi},T)\|_\omega dt\right|\\
    &\leq&T^{\frac{1}{2}}\|\varphi(\cdot;\hat{\varphi}_{T_{\varepsilon_n}}^{\varepsilon_n},
    T)-\varphi(\cdot;\tilde{\varphi},T)\|_{L^2(0,T;L^2(\Omega))}\\
    &\leq&T^{\frac{1}{2}}\|\psi(\cdot;\hat{\varphi}_{T_{\varepsilon_n}}^{\varepsilon_n},0)-\psi(\cdot;
    \tilde{\varphi},0)\|_{L^2(0,T+\delta;L^2(\Omega))}\to 0\;\;\mbox{as}\;\;n\to\infty
\end{eqnarray*}
    and
\begin{eqnarray*}
    &\;&|\langle y_0,\varphi(0;\hat{\varphi}_{T_{\varepsilon_n}}^{\varepsilon_n},T)
    -\varphi(0;\tilde{\varphi},T)\rangle_\Omega|\nonumber\\
    &\leq&\|y_0\|_\Omega\|\varphi(0,\hat{\varphi}_{T_{\varepsilon_n}}^{\varepsilon_n},T)
    -\varphi(0;\tilde{\varphi},T)\|_\Omega\\
    &=&\|y_0\|_\Omega\|\psi(T;\hat{\varphi}_{T_{\varepsilon_n}}^{\varepsilon_n},0)
    -\psi(T;\tilde{\varphi},0)\|_\Omega\to 0\;\;\mbox{as}\;\;n\to \infty.
\end{eqnarray*}
    These, along with (\ref{yu1.18}), (\ref{yu1.41}), (\ref{yu-bu-29})
    and the weakly lower
    semi-continuity of $L^2$-norm, yield
\begin{eqnarray}\label{138}
    J^T(\tilde{\varphi})&\leq&\liminf_{n\to\infty}J^T
    (\hat{\varphi}_{T_{\varepsilon_n}}^{\varepsilon_n})
    \leq
    \liminf_{n\to\infty}[J^{T_{\varepsilon_n}}_{\varepsilon_n}
    (\hat{\varphi}_{T_{\varepsilon_n}}^{\varepsilon_n})+I_n]\nonumber\\
    &\leq&\liminf_{n\to \infty}
    J^{T_{\varepsilon_n}}_{\varepsilon_n}(\hat{\varphi}_{T_{\varepsilon_n}}
    ^{\varepsilon_n})+\limsup_{n\to\infty}
    I_{n}\nonumber\\
    &\leq&\liminf_{n\to \infty}
    J^{T_{\varepsilon_n}}_{\varepsilon_n}(\hat{\varphi}_T)=J^T(\hat{\varphi}_T).
\end{eqnarray}
    Thus,
    $\tilde{\varphi}$ is also a minimizer of problem (\ref{101}).
    By the uniqueness of minimizer of this problem,
    $\tilde{\varphi}=\hat{\varphi}_T$, i.e.,  (\ref{wang1.18}) holds. Consequently,  (\ref{yu1.19}) holds.
\par
    Next, we will prove that
\begin{equation}\label{139}
     \lim_{\varepsilon\to
    0^+}J^{T_{\varepsilon}}_\varepsilon(\hat{\varphi}_{T_\varepsilon}^\varepsilon)
    =\lim_{\varepsilon\to
    0^+}J^T(\hat{\varphi}_{T_\varepsilon}^\varepsilon)=
    J^T(\hat{\varphi}_T).
\end{equation}
    From the optimality of
    $\hat{\varphi}_{T_\varepsilon}^\varepsilon$ to the problem  $(MP^\varepsilon)$ (see (\ref{106-1})),
    we have
\begin{eqnarray}\label{142}
    J^{T_\varepsilon}_\varepsilon(\hat{\varphi}_{T_\varepsilon}^\varepsilon)&\leq&
    J^{T_\varepsilon}_\varepsilon(\hat{\varphi}_T)\leq J^T(\hat{\varphi}_T)+|\langle
    y_0,\varphi^\varepsilon(0;\hat{\varphi}_T,T_\varepsilon)-\varphi
    (0;\hat{\varphi}_T,T)\rangle_\Omega|\nonumber\\
    &\;&+\frac{1}{2}\left|\left(\int_0^{T_\varepsilon}\|
    \varphi^\varepsilon(t;\hat{\varphi}_T,T_\varepsilon)\|_{\omega}dt\right)^2
    -\left(\int_0^T\|\varphi(t;\hat{\varphi}_T,T)\|_{\omega}dt\right)^2\right|.
\end{eqnarray}
   From this, we can use the similar methods  used  in the proofs of (\ref{136}) and (\ref{137}) to get
\begin{equation}\label{143}
    \limsup_{\varepsilon\to
    0^+}J^{T_\varepsilon}_\varepsilon(\hat{\varphi}_{T_\varepsilon}^\varepsilon)
    \leq\limsup_{\varepsilon\to
    0^+}J^{T_\varepsilon}_\varepsilon(\hat{\varphi}_{T_\varepsilon})\leq
    J^T(\hat{\varphi}_T).
\end{equation}
    On the other hand,  one can easily check that
\begin{eqnarray}\label{yu-bu-16}
    J^T(\hat{\varphi}_T)&\leq&
    J^T(\hat{\varphi}_{T_\varepsilon}^\varepsilon)\leq
    J^{T_\varepsilon}_{\varepsilon}(\hat{\varphi}_{T_\varepsilon}^\varepsilon)
    +|\langle
    y_0,\varphi^\varepsilon(0;\hat{\varphi}^\varepsilon_{T_\varepsilon},T_\varepsilon)
    -\varphi(0;\hat{\varphi}_{T_\varepsilon}^\varepsilon,T)\rangle_\Omega|\nonumber\\
    &\;&+\frac{1}{2}\left|\left(\int_0^{T_\varepsilon}
    \|\varphi^\varepsilon(t;\hat{\varphi}_{T_\varepsilon}^\varepsilon,
    T_\varepsilon)\|_{\omega}dt
    \right)^2
    -\left(\int_0^T\|\varphi(t;\hat{\varphi}_{T_\varepsilon}^\varepsilon,T)\|
    _{\omega}dt\right)^2\right|.
\end{eqnarray}
    From this, we also can use the very similar ways as those used in the proofs of
    (\ref{136}) and (\ref{137}) to derive that
\begin{equation}\label{144}
    J^T(\hat{\varphi}_T)\leq
    \liminf_{\varepsilon\to0^+}J^T(\hat{\varphi}_{T_{\varepsilon}}^\varepsilon)\leq
    \liminf_{\varepsilon\to
    0^+}J^{T_\varepsilon}_\varepsilon(\hat{\varphi}^\varepsilon_{T_\varepsilon}).
\end{equation}
    This, together with (\ref{143}), yields
\begin{equation}\label{yu-bu-17}
    J^T(\hat{\varphi}_T)=\lim_{\varepsilon\to 0^+}J^{T_\varepsilon}_\varepsilon(\hat{\varphi}_{T_\varepsilon}^\varepsilon).
\end{equation}
    Hence, (\ref{139}) follows (\ref{yu-bu-16}), (\ref{yu-bu-17}).

    Finally, by  (\ref{yu1.19}) and the compactness of $S(t)$, for any sequence $\{\varepsilon_n\}_{n\in\mathbb{N}}\subset\{\varepsilon\}
    _{\varepsilon\in(0,\varepsilon_0]}$ with $\varepsilon_n\to 0^+$ as $n\to \infty$,
     there exists a subsequence $\{\varepsilon_{n_k}\}_{k\in\mathbb{N}}$ of $\{\varepsilon_n\}_{n\in\mathbb{N}}$ such that
   \begin{eqnarray*}
    &\;&\|\varphi(\cdot;\hat{\varphi}^{\varepsilon_{n_k}}_{T_{\varepsilon_{n_k}}},T)
    -\varphi(\cdot;\hat{\varphi}_T,T)\|_{L^2(0,T;L^2(\Omega))}\to
    0~~\text{as}~k\to \infty
\end{eqnarray*}
    and
$$
    \|\varphi(0;\hat{\varphi}^{\varepsilon_{n_k}}_{T_{\varepsilon_{n_k}}},T)
    -\varphi(0;\hat{\varphi}_T,T)\|_\Omega\to 0~~\text{as}~k\to \infty.
$$
     These imply that
$$
    \langle
    y_0,\varphi(0;\hat{\varphi}_{T_{\varepsilon_{n_k}}}^{\varepsilon_{n_k}},T)\rangle_\Omega\to
    \langle
    y_0,\varphi(0;\hat{\varphi}_T,T)\rangle_\Omega~~\text{as}~k\to
    \infty
$$
    and
$$
    \int_0^T\|\varphi(t;\hat{\varphi}_{T_{\varepsilon_{n_k}}}^{\varepsilon_{n_k}},T)\|_{\omega}dt\to
    \int_0^T\|\varphi(t;\hat{\varphi}_T,T)\|_{\omega}dt~~\text{as}~k\to
    \infty.
$$
   Since $\{\varepsilon_n\}_{n\in\mathbb{N}}$ was arbitrarily taken, we have
$$
    \langle
    y_0,\varphi(0;\hat{\varphi}_{T_\varepsilon}^\varepsilon,T)\rangle_\Omega\to
    \langle
    y_0,\varphi(0;\hat{\varphi}_T,T)\rangle_\Omega~~\text{as}~\varepsilon\to
    0^+
$$
     and
$$
    \int_0^T\|\varphi(t;\hat{\varphi}_{T_\varepsilon}^\varepsilon,T)\|_{\omega}dt\to
    \int_0^T\|\varphi(t;\hat{\varphi}_T,T)\|_{\omega}dt~~\text{as}~\varepsilon\to
    0^+.
$$
    These, along with  (\ref{139}) and the definitions of $J^T(\hat{\varphi}_T)$ and $J^{T}(\hat{\varphi}_{T_\varepsilon}^\varepsilon)$ (see (\ref{101-1})), indicate
$$
    \|\hat{\varphi}_{T_\varepsilon}^\varepsilon\|_\Omega\to
    \|\hat{\varphi}_T\|_\Omega~~\text{as}~\varepsilon\to 0^+.
$$
   Together with  (\ref{yu1.19}), this yields
\[
    \|\hat{\varphi}_{T_\varepsilon}^\varepsilon-\hat{\varphi}_T\|_\Omega\to
    0~~\text{as}~\varepsilon\to 0^+,
\]
    and completes the proof.
\end{proof}

Now we define a functional $ J_\varepsilon^T(\cdot)$ over $L^2(\Omega)$ by setting
  \begin{eqnarray}\label{w-1-1}
    J_\varepsilon^T(\varphi_T^\varepsilon)&=&\frac{1}{2}\left
    (\int_0^T\|\varphi^\varepsilon(t;\varphi_T^\varepsilon,T)\|_\omega dt\right)^2+\langle
    y_0,\varphi^\varepsilon(0;\varphi_T^\varepsilon,T)\rangle_\Omega\nonumber\\
    &\;&+K\|\varphi_T^\varepsilon\|_\Omega,\;\;\;\varphi_T^\varepsilon\in L^2(\Omega),
\end{eqnarray}
  where $\varphi^\varepsilon(\cdot;\varphi_{T}^\varepsilon,T)$ is the solution of the equation
\begin{equation}\label{w-2-1}
\begin{cases}
    \varphi^\varepsilon_t+\triangle\varphi^\varepsilon
    +a_\varepsilon\varphi^\varepsilon=0
    &\mbox{in}\;\;\Omega\times(0,T),\\
    \varphi^\varepsilon=0&\mbox{on}\;\;\partial\Omega\times(0,T),\\
    \varphi^\varepsilon(T)=\varphi^\varepsilon_T\in L^2(\Omega).
\end{cases}
\end{equation}

  By the same way used to prove Lemma \ref{lemmawang1}, we can show that $J_\varepsilon^T(\cdot)$, with $\varepsilon>0$ sufficiently small, has a unique non-zero minimizer in $L^2(\Omega)$.
  Further, by the same way as that used in the proof of Theorem~\ref{theoremyu1.1}, one can have the following consequence.

  \begin{corollary}\label{corollaryyu1.1-1}
    Suppose that $(H_1)$ and $(H_5)$ hold. Let $\hat{\varphi}_T^\varepsilon$ be the non-zero
    minimizer of the  functional $ J_\varepsilon^T(\cdot)$ (with $\varepsilon>0$ sufficiently small) defined by (\ref{w-1-1}), and let $\hat{\varphi}_T$  be the minimizer of the functional $J^T(\cdot)$ defined by
    (\ref{101-1}).   Then $\hat{\varphi}^\varepsilon_T \to \hat{\varphi}_T$ strongly in $L^2(\Omega)$ as $\varepsilon\to 0^+$, where $\hat{\varphi}_T$ is the  minimizer of the functional $J^T(\cdot)$ defined by
    (\ref{101-1}).
\end{corollary}
\begin{remark}\label{ireyu2.1}
    Let $M_\varepsilon$ be given by  (\ref{yub2.01}) and $\hat{\varphi}_{T^*}$ be the minimizer of the functional defined by (\ref{101}). Then, it follows from Corollary~\ref{corollaryyu1.1-1} that
    \begin{equation}\label{yub2.02}
    M_\varepsilon\to \int_0^{T^*}\|\varphi(t;\hat{\varphi}_{T^*},T^*)
    \|_\omega dt\;\;\mbox{as}\;\;\varepsilon\to 0^+.
\end{equation}
\end{remark}

\section{The equivalence of minimal time and  norm  control problems}
\par
  Throughout this section, we let $T^*$  and $T^{*,1}_\varepsilon$ be accordingly
    the optimal time
    to Problems $(TP)$ and $(TP^\varepsilon_1)$.
    Define the following three admissible sets of controls:
\vskip 5pt
    \noindent$\mathcal{F}_{T^*}\equiv\{f\in
    L^\infty(\mathbb{R}^+;L^2(\Omega)):y(T^*;f,y_0)\in\overline{B_K(0)}\};$
\vskip 5pt
    \noindent$\mathcal{F}^\varepsilon_{T^*}\equiv\{f\in
    L^\infty(\mathbb{R}^+;L^2(\Omega)):y^\varepsilon
    (T^*;f,y_0)\in\overline{B_K(0)}\}$;
\vskip 5pt
    \noindent$\mathcal{F}^\varepsilon_{T^{*,1}_\varepsilon}\equiv\{f\in
    L^\infty(\mathbb{R}^+;L^2(\Omega)):y^\varepsilon
    (T^{*,1}_\varepsilon;f,y_0)\in\overline{B_K(0)}\}.$
\vskip 5pt
    \noindent Here,  $y(\cdot;f,y_0)$ and  $y^\varepsilon(\cdot;f,y_0)$ are accordingly the solutions of  the equations (\ref{equationyu2.01}) and  (\ref{equationyu2.02}), with  $u$  replaced by $f$.
    As a consequence of the
    approximate or null controllability over any interval $(0,T)$ with $T>0$ for linear parabolic  equations (see  \cite{b11, b6, b1, b12, b9, b3} and references therein),
    the admissible sets $\mathcal{F}_{T^*}$,
    $\mathcal{F}^\varepsilon_{T^*}$ and
    $\mathcal{F}^\varepsilon_{T^{*,1}_\varepsilon}$
    are  nonempty.
     We consider three minimal norm  control problems as follows:
\vskip 5pt
    \noindent$(NP_{T^*}) \;\;\;\;
    M_{T^*}\equiv \inf_{f\in\mathcal{F}_{T^*}}\{\|f\|_{L^\infty(0,T^*;L^2(\Omega))}\};$
\vskip 5pt
    \noindent$(NP_{T^*}^\varepsilon)\;\;
    \;\; M^\varepsilon_{T^*}
    \equiv \inf_{f\in\mathcal{F}^\varepsilon_{T^*}}
    \{\|f\|_{L^\infty(0,T^*;L^2(\Omega))}\}$;
\vskip 5pt
    \noindent$(NP^\varepsilon_{T^{*,1}_\varepsilon})\;
    \;\;\;M^\varepsilon_{T^{*,1}_\varepsilon}
    \equiv\inf_{f\in\mathcal{F}^\varepsilon_{T^{*,1}_\varepsilon}}
    \{\|f\|_{L^\infty(0,T^{*,1}_\varepsilon;L^2(\Omega))}\}.$
\vskip 5pt
\noindent The numbers  $M_{T^*}$, $M^\varepsilon_{T^*}$ and
    $M_{T^{*,1}_\varepsilon}^\varepsilon$ are called the minimal norms (or the optimal norms) to Problems  $(NP_{T^*})$, $(NP^\varepsilon_{T^*})$ and
    $(NP^\varepsilon_{T^{*,1}_\varepsilon})$ respectively. A control  $f_{T^*}\in\mathcal{F}_{T^*}$ is called an optimal control to  $(NP_{T^*})$ if $\|f_{T^*}\|_{L^\infty(0,T^*;L^2(\Omega))}= M_{T^*}$, and $f_{T^*}(\cdot)=0$ over $[T^*, +\infty)$. Similarly, we can define optimal controls
   $f^\varepsilon_{T^*}$ and $f^\varepsilon_{T^{*,1}_\varepsilon}$ to $(NP^\varepsilon_{T^*})$ and
    $(NP^\varepsilon_{T^{*,1}_\varepsilon})$ respectively.

\par
    Now, we define a new time optimal control problem:
\vskip5pt
    \noindent$(\overline{TP^\varepsilon})$ \;\;$\overline{T^*_\varepsilon}\equiv\inf_{u\in\mathcal{U}_{M^\varepsilon
    _{T^*}}^\varepsilon}
    \{t\in\mathbb{R}^+; y^\varepsilon(t;u,y_0)\in\overline{B_K(0)}\}$,
\vskip5pt
    \noindent where
\vskip5pt
    \noindent$\mathcal{U}^\varepsilon_{M_{T^*}^\varepsilon}\equiv\{u\in L^\infty(\mathbb{R}^+;L^2(\Omega)): u(\cdot)\in\overline{B_{M_{T^*}^\varepsilon}(0)}\;
    \mbox{over}\;\mathbb{R}^+\;\mbox{and}\;\exists t>0\;\mbox{s.t.}\;y^\varepsilon(t;u,
    y_0)\in\overline{B_K(0)}\}$.
\vskip5pt
   \noindent  By a very similar way used to prove  Proposition \ref{lemmayu3.1}, we can have that the problem
    $(\overline{TP^\varepsilon})$, with $\varepsilon\in (0,\varepsilon_\rho]$, has a unique optimal control. Moreover, this control also has the bang-bang property.

\par

  By the almost same way as that used in the proof of Theorem 1.1 in \cite{b4}, we can have the following  results.

\begin{proposition}\label{lemmayu4.2}
      Suppose that
    $(H_1)$, $(H_2)$ and $(H_3)$ hold. Let $\varepsilon_\rho$ verify (\ref{yu-b-2}). Then $(i)$ the  problems  $(TP)$ and $(NP_{T^*})$ share the same optimal control; $(ii)$ when $\varepsilon\in (0, \varepsilon_\rho]$,
    the problems $(TP^\varepsilon_1)$ and $(NP^\varepsilon_{T^{*,1}_\varepsilon})$ have the same optimal control; $(iii)$  when $\varepsilon\in (0, \varepsilon_\rho]$, the problems $(\overline{TP^\varepsilon})$ and $(NP^\varepsilon_{T^*})$ also share the some optimal control; $(iv)$ Problems $(NP_{T^*})$, $(NP_{T^*}^\varepsilon)$
    and $(NP_{T^{*,1}_\varepsilon}^\varepsilon)$, with  $\varepsilon\in (0, \varepsilon_\rho]$, have the bang-bang property and the unique optimal controls.
    \end{proposition}
\begin{proof} Since $(a)$ the problems  $(TP)$, $(TP^\varepsilon_1)$ and $(\overline{TP^\varepsilon})$ have optimal controls  (see Proposition \ref{lemmayu3.1} in our paper and Theorem 3.2 in \cite{b7}); $(b)$ the controlled equations (\ref{equationyu2.01}) and (\ref{equationyu2.02}) have the null controllability property (see \cite{b1} and references therein); and $(c)$  the problems  $(TP)$, $(TP^\varepsilon_1)$ and $(\overline{TP^\varepsilon})$ have the bang-bang property (see  Theorem 1 and Remark in and at the end of \cite{b10}), we can follow  the exactly same way as that used to prove Theorem 1.1 in \cite{b4} to  show the equivalence of minimal norm and minimal time controls stated in $(i)$, $(ii)$ and $(iii)$. We omit the detail here.

\par
   The uniqueness for the optimal controls to $(TP)$, $(TP^\varepsilon_1)$ and $(\overline{TP^\varepsilon})$ is a direct consequence of  the corresponding bang-bang property
        (see Theorem 2.1.7 on page 36 in \cite{b24} or Theorem 1.2 in \cite{b9}). Finally, the results in $(iv)$ follows at once from $(i)$, $(ii)$, $(iii)$ and the uniqueness and the bang-bang property of  the optimal controls to Problems $(TP)$, $(TP^\varepsilon_1)$ and $(\overline{TP^\varepsilon})$.  This completes the proof.
       \end{proof}

     We now  study some characteristics of the optimal controls to the problems   $(NP_{T^*})$, $(NP_{T^*}^\varepsilon)$ and $(NP^\varepsilon_{T^{*,1}_\varepsilon})$. These properties, together with
    Proposition~\ref{lemmayu4.2}, give us the corresponding characteristics for the optimal controls to the problems   $(TP)$, $(\overline{TP^\varepsilon})$ and $(TP^\varepsilon_1)$ with sufficiently small $\varepsilon$. The later will be the key in the proof of the $L^\infty$-convergence of the optimal controls stated in Theorem~\ref{ithyu1.1}. To show the above-mentioned characteristics, we will first prove the following lemma which is indeed the part $(i)$ in Theorem~\ref{ithyu1.1}.

   \begin{lemma}\label{wanglemma31}
   Under the assumptions $(H_1)$, $(H_2)$ and $(H_3)$, it holds that $\lim_{\varepsilon\rightarrow 0^+}T_\varepsilon^{*,1}=T^*$.
   \end{lemma}
   \begin{proof}
    Let $\varepsilon_\rho$ verify (\ref{yu-b-2}).   The proof will be organized in
    four steps as
    follows.
\vskip 10 pt
 \noindent   \emph{Step 1. There exists an
    $\varepsilon_K\in(0,\varepsilon_\rho]$ such that
    $\{T^{*,1}_{\varepsilon}\}_{\varepsilon\in(0,\varepsilon_K]}$
    is a bounded set.}
\par
   Because of (\ref{yu-b-1}),  there exists a time $T_{\frac{K}{2}}>0$ such
    that $\|S(T_{\frac{K}{2}})y_0\|_{\Omega}\leq \frac{K}{2}$, i.e.,
    \begin{equation}\label{WANGYUANYUAN3.1}
    y(T_{\frac{K}{2}};0,y_0)\in \overline{B_{\frac{K}{2}}(0)}.
    \end{equation}
    Meanwhile, by (\ref{yu2.03}),  there
    exists an $\varepsilon_K\in (0, \varepsilon_\rho]$ such that
\begin{equation}\label{yu3.01}
    \|y^{\varepsilon}(T_{\frac{K}{2}};0,y_0)-y(T_{\frac{K}{2}};0,y_0)\|_{\Omega}\leq
    \frac{K}{2},\;\;\mbox{when}\;\;\varepsilon\in(0,\varepsilon_K].
\end{equation}
    This, along with (\ref {WANGYUANYUAN3.1}),
    yields
\begin{equation}\label{yu3.02}
    \|y^\varepsilon(T_{\frac{K}{2}};0,y_0)\|_{\Omega}\leq K,
    \;\;\mbox{when}\;\;\varepsilon\in(0,\varepsilon_K].
\end{equation}
    From this and the optimality of $T^{*,1}_\varepsilon$ to Problem $(TP^\varepsilon_1)$, we see that  $T^{*,1}_\varepsilon\leq T_{\frac{K}{2}}$, when
    $\varepsilon\in(0,\varepsilon_K]$.

\vskip 10pt
  \noindent  \emph{Step 2. Let
    $\{\varepsilon_n\}_{n\in\mathbb{N}}\subset\{\varepsilon\}_{\varepsilon\in
    (0,\varepsilon_K]}$ be such that $\varepsilon_n\to0^+$ as $n\to\infty$. Then there are
    a subsequence of $\{\varepsilon_n\}_{n\in\mathbb{N}}$, still denoted in the same way,
    a time $\bar{T}$ and a
    control $\bar u$, with $\bar{u}(t)\in\overline{B_M(0)}$ a.e. $t\in[0,T_{\frac{K}{2}}]$, such that
    $T^{*,1}_{\varepsilon_n}\to
    \bar{T}$ and
    $u^{*,1}_{\varepsilon_n}\to \bar{u}$ weakly star in
    $L^\infty(0,T_\frac{K}{2};L^2(\Omega))$  as
    $n\to \infty$, and $y(\bar{T};\bar{u},y_0)\in \overline{B_K(0)}$.}
\par
    By the conclusion of
    Step 1, there are a time $\bar T$ and  a subsequence of $\{\varepsilon_n\}_{n\in\mathbb{N}}$, still
    denoted in the same way, such that
\begin{equation}\label{yu3.03}
    T^{*,1}_{\varepsilon_n}\to \bar{T}\;\;\mbox{as}\;\;n\to\infty.
\end{equation}
    For each $n$, we let  $u^{*,1}_{\varepsilon_n}$ be
    the optimal controls to Problem $(TP_1^{\varepsilon_n})$.
     Since $\{u^{*,1}_{\varepsilon_n}\}_{n\in\mathbb{N}}$ is   bounded  in $L^\infty(0,T_{\frac{K}{2}}; L^2(\Omega))$,
    there exist an $\bar{u}\in L^\infty(0,T_{\frac{K}{2}}; L^2(\Omega))$ and a
    subsequence of the sequence $\{\varepsilon_n\}_{n\in\mathbb{N}}$, still
    denoted in the same way, such that
\begin{equation}\label{yu3.04}
    u^{*,1}_{\varepsilon_n}\to \bar{u}\;\;\mbox{weakly star in}\;\;
    L^\infty(0,T_{\frac{K}{2}};L^2(\Omega))\;\;\mbox{as}\;\;n\to\infty.
\end{equation}
    Moreover, there exists a $C>0$ such that
\begin{eqnarray}\label{yu3.05a}
    &\;&\|y(\cdot;u_{\varepsilon_n}^{*,1},y_0)\|_{C([0,T_{\frac{K}{2}}];
    L^2(\Omega))}\nonumber\\
    &\;&\;\;\;\;+\|y^{\varepsilon_n}(\cdot;u_{\varepsilon_n}^{*,1},y_0)
    \|_{C([0,T_{\frac{K}{2}}];L^2(\Omega))}\leq
    C,\;\;\forall n\in\mathbb{N}.
\end{eqnarray}
\par
    Next, we will prove that on a subsequence of $\{\varepsilon_n\}_{n\in\mathbb{N}}$, still denoted in the same way,
\begin{equation}\label{yu3.05}
    \|y^{\varepsilon_n}(T^{*,1}_{\varepsilon_n};u^{*,1}_{\varepsilon_n},y_0)
    -y(\bar{T};\bar{u},y_0)\|_{\Omega}\to
    0\;\;\mbox{as}\;\; n\to\infty.
\end{equation}
    To show (\ref{yu3.05}), we only need to show that on a subsequence of $\{\varepsilon_n\}_{n\in\mathbb{N}}$, still denoted in the same way,
\begin{equation}\label{yu3.06}
    \|y(\bar{T};u^{*,1}_{\varepsilon_n},y_0)-y(\bar{T};\bar{u},y_0)\|_{\Omega}\to
    0\;\;\mbox{as}\;\; n\to\infty,
\end{equation}
\begin{equation}\label{yu3.07}
    \|y(T^{*,1}_{\varepsilon_n};u^{*,1}_{\varepsilon_n},y_0)
    -y(\bar{T};u^{*,1}_{\varepsilon_n},y_0)\|_{\Omega}\to
    0\;\;\mbox{as}\;\; n\to\infty
\end{equation}
    and
\begin{equation}\label{yu3.08}
    \|y^{\varepsilon_n}(T^{*,1}_{\varepsilon_n};u^{*,1}_{\varepsilon_n},y_0)
    -y(T^{*,1}_{\varepsilon_n};u^{*,1}_{\varepsilon_n},y_0)\|_{\Omega}\to
    0\;\;\mbox{as}\;\; n\to\infty.
\end{equation}
\par
   The convergence (\ref{yu3.06}) follows from (\ref{yu3.04}) and Aubin's theorem (see Theorem 1.20 on page 26 in \cite{b2}).
\par
   Now we show (\ref{yu3.07}).
    Notice that
\begin{eqnarray}\label{yuevolution3.1}
    &\;&\|y(T^{*,1}_{\varepsilon_n};u^{*,1}_{\varepsilon_n},y_0)-
    y(\bar{T};u^{*,1}_{\varepsilon_n},y_0)\|_\Omega\nonumber\\
    &=&\|S(T^{*,1}_{\varepsilon_n})y_0-S(\bar{T})y_0\|_\Omega\nonumber\\
    &\;&+\left\|\int_0^{T^{*,1}_{\varepsilon_n}}
    S(T^{*,1}_{\varepsilon_n}-t)\chi_\omega u^{*,1}_{\varepsilon_n}(t)dt-
    \int_0^{\bar{T}}S(\bar{T}-t)\chi_\omega u^{*,1}_{\varepsilon_n}(t)dt
    \right\|_\Omega\nonumber\\
    &\leq&\|S(T^{*,1}_{\varepsilon_n})y_0-S(\bar{T})y_0\|_\Omega\nonumber\\
    &\;&+\left\|\int_0^{T^{*,1}_{\varepsilon_n}\wedge\bar{T}}S(T^{*,1}_{\varepsilon_n}-t)
    \chi_\omega u^{*,1}_{\varepsilon_n}(t)dt-\int_0^{T^{*,1}_{\varepsilon_n}\wedge\bar{T}}
    S(\bar{T}-t)\chi_\omega u^{*,1}_{\varepsilon_n}(t)dt\right\|_\Omega\nonumber\\
    &\;&+\left\|\int_{T^{*,1}_{\varepsilon_n}\wedge\bar{T}}^{T^{*,1}_{\varepsilon_n}}
    S(T^{*,1}_{\varepsilon_n}-t)\chi_\omega u^{*,1}_{\varepsilon_n}(t)dt\right\|_\Omega
    +\left\|\int_{T^{*,1}_{\varepsilon_n}\wedge\bar{T}}^{\bar{T}}
    S(\bar{T}-t)\chi_\omega u^{*,1}_{\varepsilon_n}(t)dt\right\|_\Omega\nonumber\\
    &\equiv& L_n^1+L_n^2+L_n^3+L_n^4.
\end{eqnarray}
    By the strong continuity of $S(\cdot)$
    and (\ref{yu3.03}), we have
\begin{equation}\label{yuevolution3.2}
    L_n^1=\|S(T^{*,1}_{\varepsilon_n})y_0-S(\bar{T})y_0\|_\Omega\to 0\;\;\mbox{as}\;\;n\to\infty.
\end{equation}
    Meanwhile, one can easily to check that
\begin{eqnarray}\label{yuevolution3.3}
    L_n^2&=& \left\|\int_0^{T^{*,1}_{\varepsilon_n}\wedge\bar{T}}[S(T^{*,1}_{\varepsilon_n}-t)
    -S(\bar{T}-t)]\chi_\omega u^{*,1}_{\varepsilon_n}(t) dt\right\|_\Omega\nonumber\\
    &\leq&\biggl\|\int_0^{T^{*,1}_{\varepsilon_n}\wedge \bar{T}}
    [S(T^{*,1}_{\varepsilon_n}-t)-S(\bar{T}-t)]
    \chi_\omega(u^{*,1}_{\varepsilon_n}(t)-\bar{u}(t)) dt\biggl\|_\Omega\nonumber\\
    &\;&+\int_0^{T^{*,1}_{\varepsilon_n}\wedge\bar{T}}
    \|[S(T^{*,1}_{\varepsilon_n}-t)-S(\bar{T}-t)]\chi_\omega
    \bar{u}(t)\|_\Omega dt\nonumber\\
    &\equiv&L_n^{2,1}+L_n^{2,2}.
\end{eqnarray}
    By the dominated convergence theorem, we can deduce that
\begin{equation}\label{iyu3.1}
     L^{2,2}_n\to 0\;\;\mbox{as}\;\;n\to\infty.
\end{equation}
    To show that $L_n^{2,1}\to 0$ as $n\to 0$, we  first consider the following equation:
\begin{equation}\label{iyu3.2}
\begin{cases}
    \xi^n_t-\triangle \xi^n-a\xi^n=\chi_\omega(u^{*,1}_{\varepsilon_n}-\bar{u})&\mbox{in}\;
    \Omega\times(0,T_{\frac{K}{2}}),\\
    \xi^n=0&\mbox{on}\;\partial\Omega\times(0,T_{\frac{K}{2}}),\\
    \xi^n(0)=0&\mbox{in}\;\Omega.
\end{cases}
\end{equation}
    The solution for this equation in time $T^{*,1}_{\varepsilon_n}\wedge\bar{T}$
    can be written as
$$
    \xi^n(T^{*,1}_{\varepsilon_n}\wedge\bar{T})=\int_0^{T^{*,1}_{\varepsilon_n}\wedge\bar{T}}
    S(T^{*,1}_{\varepsilon_n}\wedge\bar{T}-t)\chi_\omega(u^{*,1}_{\varepsilon_n}(t)
    -\bar{u}(t))dt.
$$
    From this, (\ref{yu3.04}) and  Aubin's theorem, it follows that on a subsequence of $\{\varepsilon_n\}_{n\in\mathbb{N}}$, still denoted in the same way,
$$
    \|\xi^n(T^{*,1}_{\varepsilon_n}\wedge\bar{T})\|_\Omega\leq \|\xi^n\|_{C([0,T_{\frac{K}{2}}];L^2(\Omega))}\to 0\;\;\mbox{as}\;n\to\infty.
$$
    Hence,
\begin{eqnarray}\label{iyu3.3}
    &\;&\left\|\int_0^{T^{*,1}_{\varepsilon_n}\wedge\bar{T}}
    S(T^{*,1}_{\varepsilon_n}-t)\chi_\omega(u^{*,1}_{\varepsilon_n}(t)
    -\bar{u}(t))dt\right\|_\Omega\nonumber\\
    &\leq&\|S(T^{*,1}_{\varepsilon_n}-T^{*,1}_{\varepsilon_n}\wedge\bar{T})\|
    \|\xi^n(T^{*,1}_{\varepsilon_n}\wedge\bar{T})\|_\Omega\to 0\;\;\mbox{as}\;n\to\infty.
\end{eqnarray}
    Similarly, we can prove that on a subsequence of $\{\varepsilon_n\}_{n\in\mathbb{N}}$, still denoted in the same way,
\begin{equation}\label{iyu3.5}
    \left\|\int_0^{T^{*,1}_{\varepsilon_n}\wedge\bar{T}}
    S(\bar{T}-t)\chi_\omega(u^{*,1}_{\varepsilon_n}(t)
    -\bar{u}(t))dt\right\|_\Omega\to 0\;\;\mbox{as}\;n\to\infty.
\end{equation}
    This, together with (\ref{iyu3.3}), leads to $L_n^{2,1}\to 0$ as $n\to\infty$.
   The later, combined with (\ref{yuevolution3.3}) and  (\ref{iyu3.1}), indicates that
\begin{equation}\label{iyu3.4}
    L_n^2\to 0\;\;\mbox{as}\;n\to\infty.
\end{equation}
    On the other hand,
\begin{equation}\label{yuevolution3.4}
    L_n^3+L_n^4\leq M(|T^{*,1}_{\varepsilon_n}-T^{*,1}_{\varepsilon_n}\wedge\bar{T}|
    +|\bar{T}-T^{*,1}_{\varepsilon_n}\wedge\bar{T}|)\to 0\;\;\mbox{as}\;\;n\to\infty.
\end{equation}
    Hence,  (\ref{yu3.07}) follows from  (\ref{yuevolution3.4}), (\ref{yuevolution3.1}), (\ref{yuevolution3.2})
    and (\ref{iyu3.4}).
\par
   Then, we  show  (\ref{yu3.08}).
    Let $\{e^{\triangle t}\}_{t\geq 0}$ be the strongly continuous
    semigroup generated by $-\triangle$ in $L^2(\Omega)$. Then,
    for any $t\in[0,T^{*,1}_{\varepsilon_n}]$,
\begin{eqnarray}\label{yuevolution3.5}
    &\;&\|y^{\varepsilon_n}(t;u^{*,1}_{\varepsilon_n},y_0)
    -y(t;u^{*,1}_{\varepsilon_n},y_0)\|_\Omega\nonumber\\
    &=&\left\|\int_0^t e^{\triangle (t-s)}a_{\varepsilon_n}
    y^{\varepsilon_n}(s;u^{*,1}_{\varepsilon_n},y_0)ds
    -\int_0^t e^{\triangle (t-s)}a y(s;u^{*,1}_{\varepsilon_n},y_0)ds
    \right\|_\Omega\nonumber\\
    &\leq&\|a_{\varepsilon_n}\|_{L^\infty(\Omega)}
    \int_0^t\|y^{\varepsilon_n}(s;u^{*,1}_{\varepsilon_n},y_0)
    -y(s;u^{*,1}_{\varepsilon_n},y_0)\|_\Omega ds\nonumber\\
    &\;&+\|a_{\varepsilon_n}-a\|_{L^\infty(\Omega)}\int_0^{T^{*,1}_{\varepsilon_n}}
    \|y(s;u^{*,1}_{\varepsilon_n},y_0)\|_\Omega ds.
\end{eqnarray}
    By $(H_1)$, (\ref{yu3.05a}) and Gronwall's inequality, we have
\begin{eqnarray}\label{yuevolution3.6}
    &\;&\|y^{\varepsilon_n}(T^{*,1}_{\varepsilon_n};u^{*,1}_{\varepsilon_n},y_0)
    -y(T^{*,1}_{\varepsilon_n};u^{*,1}_{\varepsilon_n},y_0)\|_\Omega\nonumber\\
    &\leq&\|a_{\varepsilon_n}-a\|_{L^\infty(\Omega)}
    e^{\|a_{\varepsilon_n}\|_{L^\infty(\Omega)}T^{*,1}_{\varepsilon_n}}
    \int_0^{T^{*,1}_{\varepsilon_n}}\|y(s;u^{*,1}_{\varepsilon_n},y_0)\|_\Omega ds\nonumber\\
    &\;&\to 0\;\;\mbox{as}\;\;n\to\infty.
\end{eqnarray}
    This gives (\ref{yu3.08}).
\par
    From  (\ref{yu3.06}),
    (\ref{yu3.07}) and (\ref{yu3.08}),
    we conclude that (\ref{yu3.05}) holds.
    By (\ref{yu3.05}) and the closeness of $\overline{B_K(0)}$, we get that
    $y(\bar{T};\bar{u},y_0)\in\overline{B_{K}(0)}$.
\vskip 10pt
  \noindent \emph{Step 3. $\bar{T}=T^*$
    and $\bar{u}\chi_{[0,T^*)}$, when extended by zero to $[T^*,\infty)$, is  an optimal control to  $(TP)$.}
\par
    Since $y(\bar{T};\bar{u},y_0)\in\overline{B_{K}(0)}$, it holds that
    $\bar{T}\geq T^*$. Seeking for a contradiction, we suppose that $\bar{T}>T^*$.  Then we would have that
    \begin{equation}\label{WANGHENHEN3.22}
    \tau\equiv \frac{1}{2}(\bar{T}-T^*)>0.
    \end{equation}
         Notice that
    $y_{T^*}\equiv y(T^*;u^*,y_0)\in\partial\overline{B_K(0)}$, where $\partial\overline{B_K(0)}$
    is the boundary of the set $\overline{B_K(0)}$. Clearly,
$    y(\tau;0,y_{T^*})=S(\tau)y_{T^*}$.
    Therefore,
$    \|y(\tau;0,y_{T^*})\|_{\Omega}\leq Ke^{-\delta_0 \tau}$.
    This leads to
\begin{equation}\label{yu3.18}
    y(\tau;0,y_{T^*})\in\overline{B_{Ke^{-\delta_0 \tau}}(0)}.
\end{equation}
    We note that
\begin{equation}\label{WANGHENHEN3.24}
    u^*(t)=
    0,\;\;t\in[T^*,+\infty).
\end{equation}
    Then, from (\ref{yu3.18}), it follows that
$    y(T^*+\tau;u^*,y_0)\in\overline{B_{Ke^{-\delta_0 \tau}}(0)}$.
       Let $\{\varepsilon_n\}_{n\in\mathbb{N}}$ be the subsequence given in Step 2. Then, by
        (\ref{yu2.03}) and (\ref{WANGHENHEN3.24}),  there exists an $N>0$ such
    that
\begin{equation}\label{yu3.19}
    |\bar{T}-T_{\varepsilon_n}^{*,1}|< \tau,\;\;\mbox{when}\;\;n\geq N
\end{equation}
     and
$$
    \|y^{\varepsilon_n}(T^*+\tau;u^*,y_0)-y(T^*+\tau;u^*,y_0)\|_\Omega\leq
    \frac{K}{2}(1-e^{-\delta_0 \tau}),\;\;\mbox{when}\;\;n\geq N.
$$
   Hence,
\begin{equation}\label{WANGHENHEN3.26}
    y^{\varepsilon_n}(T^*+\tau;u^*,y_0)
    \in\overline{B_{\frac{K}{2}(1+e^{-\delta_0 \tau})}(0)},\;\;\mbox{when}\;\;n\geq N.
\end{equation}
  Because $\frac{K}{2}(1+e^{-\delta_0\tau})<K$, it follows from (\ref{WANGHENHEN3.24}), (\ref{WANGHENHEN3.26}), the continuous decay property of $S^{\varepsilon_n}(\cdot)y^{\varepsilon_n}(T^*; u^*,y_0)$ and the optimality of $T_{\varepsilon_n}^{*,1}$ to $(TP_1^{\varepsilon_n})$ that
  $$
   T^{*,1}_{\varepsilon_n}<T^*+\tau.
   $$
  This, along with  (\ref{WANGHENHEN3.22}) and (\ref{yu3.19}),  indicates that
 $$
    T^{*,1}_{\varepsilon_n}<\bar{T}-\tau<T^{*,1}_{\varepsilon_n},
    \;\;\mbox{when}\;\;n\geq
    N,
$$
    which leads to  a contradiction. Hence, it holds that
    \begin{equation}\label{WANGHENHEN3.27}
        \bar{T}=T^*.
        \end{equation}
        \vskip 5pt
    Since
    $y(T^*;\bar{u},y_0)=y(\bar{T};\bar{u},y_0)\in\overline{B_K(0)}$, the control  $\bar{u}\chi_{[0,T^*)}$, when extended by zero to $[T^*,+\infty)$, is
    an optimal control to problem $(TP)$.

      \noindent
        {\it Step 4. It holds that  $\lim_{\varepsilon\rightarrow 0^+}T_\varepsilon^{*,1}=T^*$.}

    By the uniqueness of the optimal control to $(TP)$ (see Proposition~ \ref{lemmayu3.1}), we have that
    \begin{equation}\label{WANGHENHEN3.28}
    \bar{u}=u^*\;\;\mbox{in}\;\;[0,T^*).
    \end{equation}
 Since $\{\varepsilon_n\}_{n\in\mathbb{N}}$ was arbitrarily taken from $\{\varepsilon\}_{\varepsilon\in (0, \varepsilon_K]}$,   it follows from (\ref{WANGHENHEN3.27}),
 (\ref{WANGHENHEN3.28}) and the conclusions in Step 2 that
 \begin{equation}\label{WANGHUANG3.29}
 T_\varepsilon^{*,1}\rightarrow T^*\;\;\mbox{and}\;\; u^{*,1}_{\varepsilon}\rightarrow u^*\;\;\mbox{weakly-star in}\;\; L^\infty(0, T^*; L^2(\Omega)),\;\;\mbox{as}\;\;\varepsilon\rightarrow 0^+.
 \end{equation}
   This  completes the proof.
   \end{proof}

Let
    $\hat{\varphi}_{T^*}$, $\hat{\varphi}^\varepsilon_{T^*}$
    and $\hat{\varphi}^\varepsilon_{T^{*,1}_\varepsilon}$ be accordingly
    the minimizers of $J^{T^*}(\cdot)$, $J^{T^*}_\varepsilon(\cdot)$
    and
    $J^{T^{*,1}_\varepsilon}_\varepsilon(\cdot)$.
Define the following controls:
\begin{equation}\label{yu4.01}
    f_{T^*}(t)=
\begin{cases}
    \left(\displaystyle\int_0^{T^*}\|\varphi(s;\hat{\varphi}_{T^*},T^*)\|
    _{\omega}ds\right)\dfrac{\chi_\omega\varphi(t;\hat{\varphi}_{T^*},T^*)}
    {\|\varphi(t;\hat{\varphi}_{T^*},T^*)\|_{\omega}},&
    t\in[0,T^*),\\
    \;\;\;\;\;\;\;\;\;\;\;\;\;\;\;\;\;\;\;0,&t\in[T^*,+\infty),
\end{cases}
\end{equation}
\begin{equation}\label{yub4.01}
    f^\varepsilon_{T^*}(t)=
\begin{cases}
    \left(\displaystyle\int_0^{T^*}\|\varphi^\varepsilon(s;
    \hat{\varphi}^\varepsilon_{T^*},T^*)\|_{\omega}ds\right)
    \dfrac{\chi_\omega\varphi^\varepsilon(t;\hat{\varphi}^\varepsilon_{T^*},T^*)}
    {\|\varphi^\varepsilon(t;\hat{\varphi}^\varepsilon_{T^*},T^*)\|_{\omega}},
    &t\in[0,T^*),\\
    \;\;\;\;\;\;\;\;\;\;\;\;\;\;\;\;\;\;\;\;0,&t\in[T^*,+\infty),
\end{cases}
\end{equation}
    and
\begin{equation}\label{yu4.02}
    f^\varepsilon_{T^{*,1}_\varepsilon}(t)=
\begin{cases}
    \left(\displaystyle\int_0^{T^{*,1}_\varepsilon}\|
    \varphi^\varepsilon(s;\hat{\varphi}
    ^\varepsilon_{T^{*,1}_\varepsilon},T^{*,1}_\varepsilon)\|_{\omega}ds\right)
    \dfrac{\chi_\omega\varphi^\varepsilon
    (t;\hat{\varphi}^\varepsilon_{T^{*,1}_\varepsilon},T^{*,1}_\varepsilon)}
    {\|\varphi^\varepsilon(t;
    \hat{\varphi}^\varepsilon_{T^{*,1}_\varepsilon},T^{*,1}_\varepsilon)\|_{\omega}},&
    t\in[0,T^{*,1}_\varepsilon),\\
    \;\;\;\;\;\;\;\;\;\;\;\;\;\;\;\;\;\;\;\;\;\;\;\;\;0,&t\in[T^{*,1}_\varepsilon,+\infty).
\end{cases}
\end{equation}

\begin{proposition}\label{lemmayu4.1}
    Suppose that $(H_1)$, $(H_2)$ and $(H_3)$ hold.
    Then, there is an $\varepsilon_\delta>0$ such that when $\varepsilon\in (0, \varepsilon_\delta]$, the controls  $f_{T^*}$, $f^\varepsilon_{T^*}$ and
    $f^\varepsilon_{T^{*,1}_\varepsilon}$, defined by (\ref{yu4.01}), (\ref{yub4.01}) and (\ref{yu4.02}) respectively,  are  accordingly the optimal controls to the problems $(NP_{T^*})$,
    $(NP^\varepsilon_{T^*})$ and
    $(NP^\varepsilon_{T^{*,1}_\varepsilon})$.
    Consequently, when $\varepsilon\in(0,\varepsilon_\delta]$,
$$
    M_{T^*}=\int_0^{T^*}\|\varphi(t,\hat{\varphi}_{T^*},T^*)\|_{\omega}dt;\;\;\;\;
    M^\varepsilon_{T^*}=\int_0^{T^*}\|\varphi^\varepsilon
    (t,\hat{\varphi}^\varepsilon_{T^*},T^*)\|_{\omega}dt
$$
    and
$$
    M^\varepsilon_{T^{*,1}_\varepsilon}=\int_0^{T^{*,1}_\varepsilon}
    \|\varphi^\varepsilon(t,\hat{\varphi}^\varepsilon_{T^{*,1}_\varepsilon},
    T^{*,1}_\varepsilon)\|_{\omega}dt.
$$
     Here $\varphi(\cdot;\varphi_{T^*},T^*)$, $\varphi^\varepsilon(\cdot;\varphi^\varepsilon_{T^{*,1}_\varepsilon},
    T^{*,1}_\varepsilon)$ and
    $\varphi^\varepsilon(\cdot;\varphi_{T^*}^\varepsilon,T^*)$ are accordingly the solutions to the equations (\ref{103}), (\ref{104}) and (\ref{w-2}).
\end{proposition}

\begin{proof} When the target set is the origin of $L^2(\Omega)$, the corresponding results in this lemma have been prove in \cite{b4} (see Theorem 3.2 in \cite{b4}).
Our proof here is very similar to those in  \cite{b4}. For the sake of the completeness of the paper, we provide the detailed proof by following steps.
\vskip 10pt
  \noindent  \emph{Step 1. There exists an $\varepsilon_\delta>0$ such that $f_{T^*}$,
   $f^\varepsilon_{T^*}$ and $f^\varepsilon_{T^{*,1}_\varepsilon}$ are well
    defined, when $\varepsilon\in(0,\varepsilon_\delta]$.}
\par
    From the optimality of $T^*$ and the bang-bang property of the optimal control $u^*$ to the problem $(TP)$ (see \cite{b10}), we have
    $y(T^*;0,y_0)\notin \overline{B_K(0)}$. Hence, the assumption $(H_5)$, where $T$ is replaced by $T^*$, holds. Then, by Lemma \ref{lemmawang1} and Lemma
    \ref{wanglemma31}, there exists an
    $\varepsilon_\delta>0$ such that
    the functionals $J^{T^*}_\varepsilon(\cdot)$ and
    $J^{T^{*,1}_\varepsilon}_\varepsilon(\cdot)$
    have the unique minimizers $\hat{\varphi}^\varepsilon_{T^*}$ and
    $\hat{\varphi}^\varepsilon_{T^*_\varepsilon}$ in
    $L^2(\Omega)$ when $\varepsilon\in(0,\varepsilon_\delta]$. Moreover,
    $\hat{\varphi}^\varepsilon_{T^*}\neq 0$ and
    $\hat{\varphi}^\varepsilon_{T^{*,1}_\varepsilon}\neq 0$ for any
    $\varepsilon\in(0,\varepsilon_\delta]$. From the unique
    continuation property of linear parabolic equations established in  \cite{b5}, we deduce that
    $$\|\varphi(t;\hat{\varphi}_{T^*},T^*)\|_{\omega}\neq 0; \;\;\;\; \|\varphi^\varepsilon
    (t;\hat{\varphi}^\varepsilon_{T^*},T^*)\|_{\omega}\neq 0,\; t\in[0,T^*)
    $$
    and
    $$
    \|\varphi^\varepsilon(t;
    \hat{\varphi}^\varepsilon_{T^{*,1}_\varepsilon},T^{*,1}_\varepsilon)\|_{\omega}\neq 0,\;\; t\in[0,T^{*,1}_\varepsilon).
    $$
    From these and from (\ref{yu4.01}), (\ref{yub4.01}) and (\ref{yu4.02}), it follows that  $f_{T^*}$,
    $f^\varepsilon_{T^*}$ and
    $f^\varepsilon_{T^{*,1}_\varepsilon}$ are well defined and belongs to
    $L^\infty(\mathbb{R}^+;L^2(\Omega))$,
    when $\varepsilon\in(0,\varepsilon_\delta]$.
\vskip 10pt
  \noindent \emph{Step 2. $f_{T^*}\in\mathcal{F}_{T^*}$,
    $f^\varepsilon_{T^*}\in\mathcal{F}^\varepsilon_{T^*}$ and
    $f^\varepsilon_{T^{*,1}_\varepsilon}
    \in\mathcal{F}^\varepsilon_{T^{*,1}_\varepsilon}$.}
\par
   The Euler equation associated with the minimizer  $\hat{\varphi}_{T^*}$ of $J^{T^*}(\cdot)$ is as
   \begin{eqnarray}\label{yu4.03}
    &\;&\left(\int_0^{T^*}\|\varphi(t;\hat{\varphi}_{T^*},T^*)\|_{\omega}dt\right)
    \int_0^{T^*}\frac{\langle\varphi(t;\hat{\varphi}_{T^*},T^*),\varphi(t,\varphi_{T^*},T^*)
    \rangle_{\omega}}{\|\varphi(t;\hat{\varphi}_{T^*})\|_{\omega}}dt\nonumber\\
    &\;&\;\;\;\;\;+\langle
    y_0,\varphi(0;\varphi_{T^*},T^*)\rangle_{\Omega}+K\frac{\langle
    \hat{\varphi}_{T^*},\varphi_{T^*}\rangle_{\Omega}}
    {\|\hat{\varphi}_{T^*}\|_{\Omega}}=0\;\;\mbox{for
    all}\;\;\varphi_{T^*}\in L^2(\Omega).
\end{eqnarray}
    Meanwhile,  by the equations (\ref{equationyu2.01}) and (\ref{103}), we have
\begin{eqnarray}\label{yu4.04}
    &\;&\langle
    y(T^*;f_{T^*},y_0),\varphi_{T^*}\rangle_{\Omega}\nonumber\\
    &=&\langle
    y_0,\varphi(0;\varphi_{T^*},T^*)\rangle_{\Omega}+\int_0^{T^*}\langle
    f_{T^*}(t),\varphi(t;\varphi_{T^*},T^*)\rangle_{\omega}dt\;\;\mbox{for
    all}\;\;\varphi_{T^*}\in L^2(\Omega).
\end{eqnarray}
    This, together with (\ref{yu4.01}) and (\ref{yu4.03}), yields
\begin{equation}\label{yucb4.01}
    \langle
    y(T^*;f_{T^*},y_0),\varphi_{T^*}\rangle_\Omega
    =-K\frac{\langle\hat{\varphi}_{T^*},
    \varphi_{T^*}\rangle_{\Omega}}{\|\hat{\varphi}_{T^*}\|_{\Omega}}\;\;\mbox{for
    all}\;\;\varphi_{T^*}\in L^2(\Omega).
\end{equation}
    Then, it follows from (\ref{yucb4.01}) that
$$
    \|y(T^*;f_{T^*},y_0)\|_{\Omega}\leq K\;\;\mbox{and}\;\;
    y(T^*;f_{T^*},y_0)\in\overline{B_K(0)}.
$$
    Hence $f_{T^*}\in\mathcal{F}_{T^*}$. Similarly, we can show that
    $f^\varepsilon_{T^*}\in\mathcal{F}^\varepsilon_{T^*}$ and
    $f^\varepsilon_{T^{*,1}_\varepsilon}\in\mathcal{F}
    ^\varepsilon_{T^{*,1}_\varepsilon}$ for any
    $\varepsilon\in(0,\varepsilon_\delta]$.
\vskip 10 pt
  \noindent   {\it Step 3.  It holds that
  $\|f_{T^*}\|_{L^\infty(\mathbb{R}^+;L^2(\Omega))}\leq
    \|g_1\|_{L^\infty(\mathbb{R}^+;L^2(\Omega))}$ for any $g_1\in
    \mathcal{F}_{T^*}$; when
    $\varepsilon\in(0,\varepsilon_\delta]$,  $\|f^\varepsilon_{T^*}\|_{L^\infty(\mathbb{R}^+;L^2(\Omega))}\leq
    \|g_2\|_{L^\infty(\mathbb{R}^+;L^2(\Omega))}$  and
    $\|f^\varepsilon_{T^{*,1}_\varepsilon}\|_{L^\infty
    (\mathbb{R}^+;L^2(\Omega))}\leq
    \|g_3\|_{L^\infty(\mathbb{R}^+;L^2(\Omega))}$ for each  $g_2\in
    \mathcal{F}^\varepsilon_{T^*}$, $g_3\in
    \mathcal{F}^\varepsilon_{T^{*,1}_\varepsilon}$.}
\par
    Suppose that $g_1\in\mathcal{F}_{T^*}$. Then
    $\|y(T^*;g_1,y_0)\|_{\Omega}\leq K$.
    By  the equations (\ref{equationyu2.01}) and (\ref{103}), we can
    conclude that
$$
    \langle
    y(T^*;g_1,y_0),\hat{\varphi}_{T^*}\rangle_{\Omega}=\langle
    y_0,\varphi(0;\hat{\varphi}_{T^*},T^*)\rangle_{\Omega}+\int_0^{T^*}\langle
    g_1(t),\varphi(t;\hat{\varphi}_{T^*},T^*)\rangle_{\omega}dt.
$$
    This, together with (\ref{yu4.04}) and (\ref{yucb4.01}), yields
\begin{eqnarray}\label{yu4.05}
    &\;&\|f_{T^*}\|^2_{L^\infty(0,T^*;L^2(\Omega))}=\int_0^{T^*}
    \langle f_{T^*}(t),\varphi(t;\hat{\varphi}_{T^*},T^*)\rangle_\omega dt\nonumber\\
    &=&\int_0^{T^*}\langle g_1(t),\varphi(t;\hat{\varphi}_{T^*},T^*)
    \rangle_\omega dt+\langle y(T^*;f_{T^*},y_0),\hat{\varphi}_{T^*}\rangle_\Omega
    -\langle y(T^*;g_1,y_0),\hat{\varphi}_{T^*}\rangle_\Omega\nonumber\\
    &\leq&
    \int_0^{T^*}\langle
    g_1(t),\varphi(t;\hat{\varphi}_{T^*},T^*)\rangle_{\omega}dt
    -\langle y(T^*;g_1,y_0),\hat{\varphi}_{T^*}\rangle_{\Omega}
    -K\|\hat{\varphi}_{T^*}\|_{\Omega}\nonumber\\
    &\leq& \int_0^{T^*}\langle
    g_1(t),\varphi(t;\hat{\varphi}_{T^*},T^*)\rangle_{\omega}dt\nonumber\\
    &\leq&\|g_1\|_{L^\infty(\mathbb{R}^+;L^2(\Omega))}
    \int_0^{T^*}\|\varphi(t;\hat{\varphi}_{T^*},T^*)\|_{\omega}dt\nonumber\\
    &=&\|g_1\|_{L^\infty(\mathbb{R}^+;L^2(\Omega))}
    \|f_{T^*}\|_{L^\infty(\mathbb{R}^+;L^2(\Omega))}.
\end{eqnarray}
   Now, (\ref{yu4.05}) leads to $\|f_{T^*}\|_{L^\infty(\mathbb{R}^+;L^2(\Omega))}\leq
    \|g_1\|_{L^\infty(\mathbb{R}^+;L^2(\Omega))}$ for each $g_1\in\mathcal{F}_{T^*}$. Similarly, we can prove
    that
$$
    \|f^\varepsilon_{T^*}\|_{L^\infty(\mathbb{R}^+;L^2(\Omega))}\leq
    \|g_2\|_{L^\infty(\mathbb{R}^+;L^2(\Omega))}\;\;\mbox{for
    any}\;\;g_2\in\mathcal{F}^\varepsilon_{T^*}
$$
    and
$$
    \|f^\varepsilon_{T^{*,1}_\varepsilon}\|
    _{L^\infty(\mathbb{R}^+;L^2(\Omega))}\leq
    \|g_3\|_{L^\infty(\mathbb{R}^+;L^2(\Omega))}\;\;\mbox{for
    any}\;\;g_3\in\mathcal{F}_{T^{*,1}_\varepsilon}^\varepsilon,
$$
     when $\varepsilon\in(0,\varepsilon_\delta]$. The
     proof is completed.
\end{proof}

The following is a consequence of Proposition \ref{lemmayu4.2} and Proposition \ref{lemmayu4.1}.

\begin{corollary}\label{remarkyu-e-1}
   Suppose that
    $\varepsilon\in(0,\varepsilon_\rho\wedge\varepsilon_\delta]$. Let  $M_\varepsilon$  be given by  (\ref{yub2.01}). Let $\hat{\varphi}_{T^*}$ and
    $\hat{\varphi}^\varepsilon_{T^{*,1}_\varepsilon}$ be the minimizers
    of $J^{T^*}(\cdot)$  and
    $J^{T^{*,1}_\varepsilon}_\varepsilon(\cdot)$ respectively, and let $u^*$ and $u^{*,1}_\varepsilon$ be the optimal controls to $(TP)$ and $(TP_1^\varepsilon)$. Then it holds that
    \begin{equation}
    M_\varepsilon\equiv M_{T^*}^\varepsilon,
    \end{equation}
    \begin{equation}\label{yu4.06}
    M=M_{T^*}=M^\varepsilon_{T^{*,1}_\varepsilon}
    =\int_0^{T^*}\|\varphi(t;\hat{\varphi}_{T^*},T^*)\|_{\omega}dt
    =\int_0^{T^{*,1}_\varepsilon}\|\varphi^\varepsilon
    (t;\hat{\varphi}^\varepsilon_{T^{*,1}_\varepsilon},T^{*,1}_\varepsilon)\|_{\omega}dt,
\end{equation}
   \begin{equation}\label{yu4.07}
    u^*(t)=
\begin{cases}
    M\dfrac{\chi_\omega\varphi
    (t;\hat{\varphi}_{T^*},T^*)}{\|\varphi(t;\hat{\varphi}_{T^*},T^*)\|_{\omega}},
    &t\in[0,T^*),\\
    ~~~~~~~~0,&t\in[T^*,\infty)
\end{cases}
\end{equation}
and
\begin{equation}\label{yu4.08}
    u^{*,1}_\varepsilon(t)
    =
\begin{cases}
    M\dfrac{\chi_\omega\varphi^\varepsilon(t;
    \hat{\varphi}^\varepsilon_{T^{*,1}_\varepsilon},T^{*,1}_\varepsilon)}
    {\|\varphi^\varepsilon(t;
    \hat{\varphi}^\varepsilon_{T^{*,1}_\varepsilon},T^{*,1}_\varepsilon)\|_{\omega}},
    &t\in[0,T^{*,1}_\varepsilon),\\
    ~~~~~~~~~0,&t\in[T^{*,1}_\varepsilon,\infty).
\end{cases}
\end{equation}
\end{corollary}

\section{The proofs of Theorem \ref{ithyu1.1} and Theorem \ref{ithyu1.2}}

    In this section, we are going to prove Theorem \ref{ithyu1.1} and Theorem \ref{ithyu1.2}.

\begin{proof}[Proof of Theorem \ref{ithyu1.1}]
    The part $(i)$ has proved in Lemma \ref{wanglemma31}.
\par
    Now, we prove the part $(ii)$.  Let $\varepsilon_K>0$ be the number given in Step 1 of the proof of Lemma \ref{wanglemma31}.
    Given $\eta\in(0,T^*)$, by the conclusion in part $(i)$,     there exists an
    $\varepsilon_\eta=\varepsilon(\eta)\in(0,\varepsilon_K]$
    such that
$$
    |T^{*,1}_\varepsilon-T^*|\leq
    \eta,\;\;\mbox{when}\;\;\varepsilon\in(0,\varepsilon_\eta].
$$
    From (\ref{WANGHUANG3.29}),  we have
$$
    u^{*,1}_{\varepsilon}\to u^*\;\;\mbox{weakly star in}\;\;
    L^\infty(0,T^*-\eta;L^2(\Omega))\;\;\mbox{as}\;\; \varepsilon\to 0^+.
$$
    Hence, we can get that
\begin{equation}\label{WANGHENHEN4.1}
    u^{*,1}_\varepsilon\to u^*\;\;\mbox{weakly in}\;\;
    L^2(0,T^*-\eta; L^2(\Omega))\;\;\mbox{as}\;\; \varepsilon\to0^+.
\end{equation}
   On the other hand,  these optimal controls have   the bang-bang property (see \cite{b10}), i.e.,
\begin{equation}\label{WANGHENHEN4.2}
    \|u^*(t)\|_{\Omega}=M,\;\;\forall t\in[0,T^*-\eta]
\end{equation}
    and
\begin{equation}\label{WANGHENHEN4.3}
    \|u^{*,1}_\varepsilon(t)\|_{\Omega}=M,\;\;\forall t\in[0,T^*-\eta],\;\;\mbox{when}\;\;\varepsilon\in(0,\varepsilon_\eta].
\end{equation}
     Now, it follows from (\ref{WANGHENHEN4.1}), (\ref{WANGHENHEN4.2}) and (\ref{WANGHENHEN4.3}) that
      \begin{equation}\label{WANGHENHEN4.4}
     u^{*,1}_\varepsilon\rightarrow  u^*\;\;\mbox{strongly in }\;\; L^2(0,T^*-\eta; L^2(\Omega)),\;\;\mbox{as}\;\; \varepsilon\rightarrow 0^+.
      \end{equation}
     Since $\eta>0$ is arbitrarily and because $\|u^*(t)\|_\Omega\leq M$ and $\|u^{*,1}_\varepsilon(t)\|_\Omega\leq M$ for a.e. $t\in \mathbb{R}^+$, it follows from (\ref{WANGHENHEN4.4}) that
      $$
        u^{*,1}_\varepsilon\rightarrow  u^*\;\;\mbox{strongly in }\;\; L^2(0,T^*; L^2(\Omega)),\;\;\mbox{as}\;\; \varepsilon\rightarrow 0^+.
        $$
        This completes the proof of the part $(ii)$.
\par
    Finally, we prove the part  $(iii)$.
    By the conclusion in the part $(i)$ of this theorem, one can easily
    check that the assumption $(H_4)$, where  $T_\varepsilon=T^{*,1}_\varepsilon$ and $T=T^*$ holds. Meanwhile, by the optimality
    and bang-bang property of
    $u^*$, we know that $y(T^*;0,y_0)\notin\overline{B_K(0)}$. Thus,
    $(H_5)$, where $T=T^*$  holds. Hence, we can apply Theorem \ref{theoremyu1.1}, with $T=T^*$ and $T_\varepsilon=T^{*,1}_\varepsilon$, to get
\begin{equation}\label{yu-bu-4-12-18}
    \hat{\varphi}_{T^{*,1}_\varepsilon}^\varepsilon\to\hat{\varphi}_{T^*}\;\;
    \mbox{strongly in}\;\;L^2(\Omega)\;\;
    \mbox{as}\;\; \varepsilon\to 0^+,
\end{equation}
     where $\hat{\varphi}_{T^*}$ and
    $\hat{\varphi}_{T^{*,1}_\varepsilon}^\varepsilon$
    are accordingly the minimizers of $J^{T^*}(\cdot)$ and
    $J_\varepsilon^{T^{*,1}_\varepsilon}(\cdot)$ defined by (\ref{101}) and (\ref{102})
    respectively.
\par
   Given an $\eta>0$, by the conclusion in the part $(i)$,  there exists an $\varepsilon_\eta\in(0, \varepsilon_\rho\wedge\varepsilon_\delta] $, where $\varepsilon_\rho$ verifies (\ref{yu-b-2}) and
   $\varepsilon_\delta$ is given by Proposition~\ref{lemmayu4.1},
   such that
    \begin{equation}\label{WANGGENG4.6}
    T^{*,1}_\varepsilon> T^*-\eta,\;\;\mbox{when} \;\;\varepsilon\in(0,\varepsilon_\eta].
    \end{equation}
      We  claim that there exists a $C_\eta>0$
    such that
\begin{equation}\label{yu4.44}
    \|\varphi(t;\hat{\varphi}_{T^*},T^*)\|_\omega\geq
    C_\eta\;\;\mbox{for all}\;\;t\in[0,T^*-\eta],
\end{equation}
   where $\varphi(\cdot;\hat{\varphi}_{T^*},T^*)$ is the solution to equation (\ref{103}).

   Indeed, if the above did not hold, then there would be a sequence $\{t_n\}_{n\in \mathbb{N}}\in[0,T^*-\eta]$ such that
   $$
    \|\varphi(t_n;\hat{\varphi}_{T^*},T^*)\|_\omega<\frac{1}{n}.
$$
   Without loss of generality, we can assume that $t_n\to \hat{t}\in [0,T^*-\eta]$ as $n\to \infty$. This, along with the above inequality, yields
   $$
    \|\varphi(\hat{t};\hat{\varphi}_{T^*},T^*)\|_\omega=0.
$$
Then, by the unique continuation property established in  \cite{b5}, it holds that
 $\hat{\varphi}_{T^*}=0$, which leads to a contradiction. Hence, (\ref{yu4.44}) stands.

Now by Corollary~\ref{remarkyu-e-1} (see (\ref{yu4.07}) and (\ref{yu4.08})) and by  (\ref{WANGGENG4.6}) and (\ref{yu4.44}),  we see that when $\varepsilon\in(0,\varepsilon_\eta]$ and $t\in[0,T^*-\eta]$,
\begin{eqnarray}\label{yu4.45}
    &\;&\frac{1}{M^2}\|u^{*,1}_{\varepsilon}(t)-u^*(t)\|^2_\Omega
    =\left\|\frac{\varphi^{\varepsilon}
    (t;\hat{\varphi}_{T^{*,1}_{\varepsilon}}^{\varepsilon},T^{*,1}_{\varepsilon})}
    {\|\varphi^{\varepsilon}(t;\hat{\varphi}_{T^{*,1}_{\varepsilon}}
    ^{\varepsilon},T^{*,1}_{\varepsilon})\|_\omega}
    -\frac{\varphi(t;\hat{\varphi}_{T^*},T^*)}{\|\varphi
    (t;\hat{\varphi}_{T^*},T^*)\|_\omega}\right\|_\omega^2\nonumber\\
    &=&2+2\left[\frac{\langle\varphi^{\varepsilon}
    (t;\hat{\varphi}_{T^{*,1}_{\varepsilon}}^{\varepsilon},T^{*,1}_{\varepsilon}),
    \varphi^{\varepsilon}(t;
    \hat{\varphi}_{T^{*,1}_{\varepsilon}}^{\varepsilon},T^{*,1}_{\varepsilon})
    -\varphi(t;\hat{\varphi}_{T^*},T^*)\rangle_\omega}
    {\|\varphi^{\varepsilon}(t;
    \hat{\varphi}_{T^{*,1}_{\varepsilon}}^{\varepsilon},T^{*,1}_{\varepsilon})\|_\omega
    \|\varphi(t;\hat{\varphi}_{T^*},T^*)\|_\omega}-\frac{\|\varphi^{\varepsilon}
    (t;\hat{\varphi}_{T^{*,1}_{\varepsilon}}^{\varepsilon},T^{*,1}_{\varepsilon})\|_\omega}
    {\|\varphi(t;\hat{\varphi}_{T^*},T^*)\|_\omega}\right]\nonumber\\
    &\leq&2\left[\frac{\|\varphi^{\varepsilon}
    (t;\hat{\varphi}_{T^{*,1}_{\varepsilon}}^{\varepsilon},T^{*,1}_{\varepsilon})
    -\varphi(t;\hat{\varphi}_{T^*},T^*)\|_\omega}{\|
    \varphi(t;\hat{\varphi}_{T^*},T^*)\|_\omega}
    -\frac{\|\varphi^{\varepsilon}(t;
    \hat{\varphi}_{T^{*,1}_{\varepsilon}}^{\varepsilon},T^{*,1}_{\varepsilon})\|_\omega
    -\|\varphi(t;\hat{\varphi}_{T^*},T^*)\|_\omega}{
    \|\varphi(t;\hat{\varphi}_{T^*},T^*)\|_\omega}\right]\nonumber\\
    &\leq&4\frac{\|\varphi^{\varepsilon}
    (t;\hat{\varphi}_{T^{*,1}_{\varepsilon}}^{\varepsilon},T^{*,1}_{\varepsilon})
    -\varphi(t;\hat{\varphi}_{T^*},T^*)\|_\omega}{\|
    \varphi(t;\hat{\varphi}_{T^*},T^*)\|_\omega}
    \leq \frac{4}{C_\eta}\|\varphi^{\varepsilon}
    (t;\hat{\varphi}_{T^{*,1}_{\varepsilon}}^{\varepsilon},T^{*,1}_{\varepsilon})
    -\varphi(t;\hat{\varphi}_{T^*},T^*)\|_\omega.
\end{eqnarray}
   Because of (\ref{yu-bu-4-12-18}),  we have
\begin{eqnarray}\label{yu-bu-411}
    &\;&\sup_{t\in[0,T^*-\eta]}\|\varphi(t;\hat{\varphi}_{T^{*,1}_\varepsilon}
    ^\varepsilon,T^*)-\varphi(t;\hat{\varphi}_{T^*},T^*)\|_\Omega\nonumber\\
    &=&\sup_{t\in[0,T^*-\eta]}\|S(T^*-t)(\hat{\varphi}_{T^{*,1}_\varepsilon}
    ^\varepsilon-\hat{\varphi}_{T^*})\|_\Omega
    \leq\|\hat{\varphi}_{T^{*,1}_\varepsilon}
    ^\varepsilon-\hat{\varphi}_{T^*}\|_\Omega\to 0\;\;\mbox{as}\;\;\varepsilon\to 0^+.
\end{eqnarray}
   By the strong continuity of $S(\cdot)$ and the fact that $T^{*,1}_\varepsilon\to T^*$ as $\varepsilon\to 0^+$, we have
\begin{eqnarray}\label{yu-bu-412}
    &\;&\sup_{t\in[0,T^*-\eta]}\|\varphi(t;\hat{\varphi}_{T^{*,1}_\varepsilon}^\varepsilon,
    T^{*,1}_\varepsilon)-\varphi(t;\hat{\varphi}_{T^{*,1}_\varepsilon}^\varepsilon,T^*)\|
    _\Omega\nonumber\\
    &=&\sup_{t\in[0,T^*-\eta]}\|[S({T^{*,1}_\varepsilon}-t)-S(T^*-t)]
    \hat{\varphi}_{T^{*,1}_\varepsilon}^\varepsilon\|_\Omega\nonumber\\
    &\leq&\sup_{t\in[0,T^*-\eta]}\|[S(T^{*,1}_\varepsilon-t)-S(T^*-t)]
    (\hat{\varphi}^\varepsilon_{T^{*,1}_\varepsilon}-\hat{\varphi}_{T^*})\|_\Omega\nonumber\\
    &\;&+\sup_{t\in[0,T^*-\eta]}\|[S(T^{*,1}_\varepsilon-t)-S(T^*-t)]\hat{\varphi}_{T^*}\|
    _\Omega\nonumber\\
    &\leq&2\|\hat{\varphi}_{T^{*,1}_\varepsilon}^\varepsilon-\hat{\varphi}_{T^*}\|_\Omega
    +\|[S(T^{*,1}_\varepsilon-T^*+\eta)-S(\eta)]\hat{\varphi}_{T^*}\|_\Omega\nonumber\\
    &\;&\to 0\;\;\mbox{as}\;\;\varepsilon\to 0^+.
\end{eqnarray}
  Let  $\zeta^\varepsilon(\cdot)=\varphi^\varepsilon(\cdot;\hat{\varphi}
    _{T^{*,1}_\varepsilon}^\varepsilon,T^{*,1}_\varepsilon)-\varphi(\cdot;\hat{\varphi}
    _{T^{*,1}_\varepsilon}^\varepsilon,T^{*,1}_\varepsilon)$. Then it holds that
\begin{equation}\nonumber
\begin{cases}
    \zeta^\varepsilon_t+\triangle\zeta^\varepsilon
    +a_\varepsilon\zeta^\varepsilon+(a_\varepsilon-a)\varphi_\varepsilon
    =0&\mbox{in}\;\;\Omega\times
    (0,T^{*,1}_\varepsilon),\\
    \zeta^\varepsilon=0&\mbox{on}\;\;\partial\Omega\times(0,T^{*,1}_\varepsilon),\\
    \zeta^\varepsilon(T^{*,1}_\varepsilon)=0&\mbox{in}\;\;\Omega,
\end{cases}
\end{equation}
    where $\varphi_\varepsilon(\cdot)=\varphi(\cdot;\hat{\varphi}
    _{T^{*,1}_\varepsilon}^\varepsilon,T^{*,1}_\varepsilon)$. It is obvious that
    $$\zeta^\varepsilon(t)=-\int_t^{T^{*,1}_\varepsilon}S^\varepsilon
    (T^{*,1}_\varepsilon-s)(a_\varepsilon-a)\varphi
    _\varepsilon ds,\;\; t\in [0,T^*-\eta]$$ and
\begin{eqnarray}\label{yu-bu-413}
    \|\zeta^\varepsilon(t)\|_\Omega\leq \|a_\varepsilon-a\|_{L^\infty(\Omega)}
    T^{*,1}_\varepsilon\|\varphi(\cdot;\hat{\varphi}_{T^{*,1}_\varepsilon}
    ^\varepsilon,T^{*,1}_\varepsilon)\|_{C([0,T^{*,1}_\varepsilon];L^2(\Omega))}.
\end{eqnarray}
    However,
\begin{eqnarray*}
    \sup_{t\in[0,T^{*,1}_\varepsilon]}\|\varphi(t;\hat{\varphi}_{T^{*,1}_\varepsilon}
    ^\varepsilon,T^{*,1}_\varepsilon)\|_\Omega&=&\sup_{t\in[0,T^{*,1}_\varepsilon]}\left
    \|\int_t^{T^{*,1}_\varepsilon}S(T^{*,1}_\varepsilon-s)\hat{\varphi}^\varepsilon
    _{T^{*,1}_\varepsilon}ds\right\|_{\Omega}\nonumber\\
    &\leq&\|\hat{\varphi}_{T^{*,1}_\varepsilon}^\varepsilon\|_\Omega T^{*,1}_\varepsilon\to
    \|\hat{\varphi}_{T^*}\|_\Omega T^*\;\;\mbox{as}\;\;\varepsilon\to 0^+.
\end{eqnarray*}
    This, together with (\ref{yu-bu-413}) and $(H_1)$, gives
\begin{equation}\label{yu-bu-414}
    \|\varphi^\varepsilon(\cdot;\hat{\varphi}
    _{T^{*,1}_\varepsilon}^\varepsilon,T^{*,1}_\varepsilon)-\varphi(\cdot;\hat{\varphi}
    _{T^{*,1}_\varepsilon}^\varepsilon,T^{*,1}_\varepsilon)\|_{C([0,T^*-\eta];L^2(\Omega))}
    \to 0\;\;\mbox{as}\;\;\varepsilon\to 0^+.
\end{equation}
    Therefore, it follows from (\ref{yu-bu-411}), (\ref{yu-bu-412}) and (\ref{yu-bu-414}) that
\begin{eqnarray}\label{yu4.46}
    &\;&\sup_{t\in[0,T^*-\eta]}\|\varphi^{\varepsilon}
    (t;\hat{\varphi}^{\varepsilon}_{T^{*,1}_{\varepsilon}},T^{*,1}_{\varepsilon})
    -\varphi(t;\hat{\varphi}_{T^*},T^*)\|_\Omega\nonumber\\
    &\leq&\|\varphi^\varepsilon(\cdot;\hat{\varphi}_{T^{*,1}_\varepsilon}^\varepsilon,
    T^{*,1}_\varepsilon)-\varphi(\cdot;\hat{\varphi}^\varepsilon_{T^{*,1}_\varepsilon};
    T^{*,1}_\varepsilon)\|_{C([0,T^*-\eta];L^2(\Omega))}\nonumber\\
    &\;&+\|\varphi(\cdot;\hat{\varphi}^\varepsilon_{T^{*,1}_\varepsilon};T^{*,1}_\varepsilon)
    -\varphi(\cdot;\hat{\varphi}^\varepsilon_{T^{*,1}_\varepsilon};T^*)
    \|_{C([0,T^*-\eta];L^2(\Omega))}\nonumber\\
    &\;&+\|\varphi(\cdot;\hat{\varphi}^\varepsilon_{T^{*,1}_\varepsilon};T^*)
    -\varphi(\cdot;\hat{\varphi}_{T^*};T^*)\|_{C([0,T^*-\eta];L^2(\Omega))}
    \to
    0\;\;\mbox{as}\;\;\varepsilon\to0^+.
\end{eqnarray}
   By (\ref{yu4.46}) and (\ref{yu4.45}), we see that
\begin{equation}\label{yu4.47}
    \sup_{t\in[0,T^*-\eta]}\|u^{*,1}_{\varepsilon}(t)
    -u^*(t)\|_\Omega\to0\;\;\mbox{as}\;\;\varepsilon\to0^+.
\end{equation}
   This completes the proof.
\end{proof}

\par

\begin{proof}[Proof of Theorem \ref{ithyu1.2}]
   We first prove
    the part $(i)$. Note that the semigroup $S^\varepsilon(\cdot)$ is analytic.
   Thus,  from  \cite{b10} (see Theorem 1 and Remark in and at the end of this paper), it follows that when $\varepsilon\in(0,\varepsilon_\rho\wedge\varepsilon_\delta]$ (where $\varepsilon_\rho$ verifies (\ref{yu-b-2}) and $\varepsilon_\delta$ is given in Proposition \ref{lemmayu4.1} respectively),
\begin{equation}\label{yub4.11}
    \|u^{*,2}_\varepsilon(t)\|_{\Omega}=M_\varepsilon\;
    \;\mbox{a.e.}\;\;t\in[0,T^{*,2}_\varepsilon).
\end{equation}
    Let $f_{T^*}^\varepsilon$ be the optimal control to Problem $(NP_{T^*}^\varepsilon)$.
    By (\ref{yub2.01}) and (\ref{yub4.01}), we have
   \begin{equation}\label{WANGHENHEN4.16}
   M_{T^*}^\varepsilon=\|f^\varepsilon_{T^*}(t)\|_{\Omega}=\int_0^T\|\varphi^\varepsilon
   (t;\hat{\varphi}_{T^*}^\varepsilon,T^*)\|_\omega dt=
   M_\varepsilon\;\;\mbox{a.e.}\;\;t\in[0,T^*),
   \end{equation}
    where $M_{T^*}^\varepsilon$ is the optimal norm to Problem $(NP_{T^*}^\varepsilon)$
    and $\hat{\varphi}_{T^*}^\varepsilon$ is the minimizer of (\ref{w-1}).
    By the
    optimality of $f^\varepsilon_{T^*}$ to the problem
    $(NP_{T^*}^\varepsilon)$, we get
$$
    y^\varepsilon(T^*;f^\varepsilon_{T^*},y_0)\in\overline{B_K(0)}.
$$
    This leads to
\begin{equation}\label{yub4.12}
    T^{*,2}_\varepsilon\leq T^*\;\;\mbox{for
    all}\;\;\varepsilon\in(0,\varepsilon_\rho\wedge\varepsilon_\delta].
\end{equation}
Seeking for a contradiction, we suppose that there did exist an
    $\bar{\varepsilon}\in(0,\varepsilon_\rho\wedge\varepsilon_\delta]$ such that
\begin{equation}\label{yub4.13}
    T^{*,2}_{\bar{\varepsilon}}< T^*.
\end{equation}
   Let $u^{*,2}_{\bar{\varepsilon}}$ be the optimal control to Problem $(TP_2^{\bar{\varepsilon}})$.
   Then
\begin{equation}\label{yub4.14}
    u^{*,2}_{\bar{\varepsilon}}(\cdot)=
    0\;\;\;\mbox{in}\;\;[T^{*,2}_{\bar{\varepsilon}},+\infty),
\end{equation}
   By (\ref{yu-b-2}), (\ref{yub4.14}) and the optimality of $u^{*,2}_{\bar{\varepsilon}}$ to the problem $(TP^\varepsilon_2)$, we have
\begin{eqnarray}\label{yub4.16}
    \|y^{\bar{\varepsilon}}(T^*;u^{*,2}_{\bar{\varepsilon}},y_0)\|_\Omega&=&
    \|y^{\bar{\varepsilon}}(T^*-T^{*,2}_{\bar{\varepsilon}};0,y^{\bar{\varepsilon}}
    (T^{*,2}_{\bar{\varepsilon}};u^{*,2}_{\bar{\varepsilon}},y_0))\|_\Omega\nonumber\\
    &\leq&e^{-\hat{\delta}(T^*-T^{*,2}_{\bar{\varepsilon}})}\|y^{\bar{\varepsilon}}
    (T^{*,2}_{\bar{\varepsilon}};u^{*,2}_{\bar{\varepsilon}},y_0)\|_\Omega<K.
\end{eqnarray}
   This implies
$$
    y^{\bar{\varepsilon}}(T^*;u^{*,2}_{\bar{\varepsilon}},y_0)\in\overline{B_K(0)}.
$$
    Thus it holds that
    $u^{*,2}_{\bar{\varepsilon}}\in\mathcal{F}^{\bar{\varepsilon}}_{T^*}$. By (\ref{yub4.11}), (\ref{WANGHENHEN4.16}) and (\ref{yub4.14}), we have
\begin{equation}\label{yub4.15}
    M_{\bar{\varepsilon}}=M^{\bar{\varepsilon}}_{T^*}=\|u^{*,2}_{\bar{\varepsilon}}
    \|_{L^\infty(\mathbb{R}^+;L^2(\Omega))}=
    \|f^{\bar{\varepsilon}}_{T^*}\|_{L^\infty(\mathbb{R}^+;L^2(\Omega))}.
\end{equation}
    By the uniqueness of
    $f^{\bar{\varepsilon}}_{T^*}$ (see Proposition \ref{lemmayu4.2}), we get
\begin{equation}\label{yub4.18}
    u^{*,2}_{\bar{\varepsilon}}(t)=f^{\bar{\varepsilon}}_{T^*}(t)
    \;\;\mbox{a.e.}\;\;t\in[0,T^*).
\end{equation}
    Then, from the definition of $u^{*,2}_{\bar{\varepsilon}}$ (see (\ref{yub4.11}) and
    (\ref{yub4.14})),
$$
    f^{\bar{\varepsilon}}_{T^*}\equiv0\;\;\mbox{in}\;\; [T^{*,2}_{\bar{\varepsilon}},T^*].
$$
    It contradicts to the definition of
    $f^{\bar{\varepsilon}}_{T^*}$ in (\ref{yub4.01}). Therefore
$$
    T^{*,2}_\varepsilon\equiv T^*,\;\;\mbox{when}\;\; \varepsilon\in(0,\varepsilon_\rho\wedge\varepsilon_\delta].
$$
    Let $\varepsilon_\rho\wedge\varepsilon_\delta=\varepsilon_0$, we complete the proof of part $(i)$.
\par
    Now we give the proof of part $(ii)$.
     We note that $M_\varepsilon\equiv M_{T^*}^\varepsilon$ for each $\varepsilon\in(0,\varepsilon_0]$. It follows that the problems $(TP_2^\varepsilon)$
    and $(\overline{TP^\varepsilon})$ defined in Section 3 are the same for each $\varepsilon\in(0,\varepsilon_0]$. Hence,
    $(TP_2^\varepsilon)$ (i.e., $(\overline{TP^\varepsilon})$) and
    $(NP_{T^*}^\varepsilon)$ share the same optimal control (see Proposition
    \ref{lemmayu4.2}) for each $\varepsilon\in(0,\varepsilon_0]$. By the definition of $f^\varepsilon_{T^*}$ (see (\ref{yub4.01})), we get the formula to $u^{*,2}_\varepsilon$. This gives the conclusion of part $(ii)$.
\par
    For the proof of part $(iii)$, we note that, by the definition of $M_\varepsilon$ (see (\ref{yub2.01})),
\begin{equation}\nonumber
    M_\varepsilon=\int_0^{T^*}\|\varphi^\varepsilon(t; \hat{\varphi}_{T^*}^\varepsilon,T^*)\|_{\omega}dt,
\end{equation}
    where $\hat{\varphi}_{T^*}^\varepsilon$ is the minimizer of (\ref{w-1}). From (\ref{yub2.02}),
    we have $M_\varepsilon\to M$ as $\varepsilon\to 0^+$. This completes the proof of part $(iii)$.
\par
    Next, we prove the conclusion of part $(iv)$. Since the admissible control set $\mathcal{U}_{M_\varepsilon}^\varepsilon$ is a bounded set in $L^\infty(\mathbb{R}^+;L^2(\Omega))$ (note that $M_\varepsilon\to M$ as $\varepsilon\to 0^+$), we arbitrarily take a sequence $\{\varepsilon_n\}_{n\in\mathbb{N}}\subset\{\varepsilon\}_{\varepsilon\in(0,\varepsilon_0]}$ such that $\varepsilon_n\to 0^+$ as $n\to\infty$, there exists a subsequence of $\{\varepsilon_n\}_{n\in\mathbb{N}}$, still denoted in the same way,  and $\tilde{u}\in L^\infty(0,T^*;L^2(\Omega))$ such that
\begin{equation}\label{yu-12-24-1}
    u^{*,2}_{\varepsilon_n}\to\tilde{u}\;\;\mbox{weakly star in}\;\;L^\infty(0,T^*;L^2(\Omega))\;\;\mbox{as}\;\;n\to\infty.
\end{equation}
    It follows from Ascoli's theorem and Aubin's theorem that there exists a subsequence of $\{\varepsilon_n\}_{n\in\mathbb{N}}$, still denoted in the way, such that
$$
    \|y(T^*;u^{*,2}_{\varepsilon_n},y_0)-y(T^*;\tilde{u},y_0)\|_\Omega\to 0\;\;as\;\;n\to\infty.
$$
    Because of
$$
    y(T^*;u^{*,2}_{\varepsilon_n},y_0)=y(T^{*,2}_{\varepsilon_n};u^{*,2}_{\varepsilon_n},y_0)
    \in\overline{B_K(0)},
$$
    we have
$$
    y(T^*;\tilde{u},y_0)\in \overline{B_K(0)}.
$$
    But
$$
    \|u^{*,2}_\varepsilon\|_{L^\infty(0,T^*;L^2(\Omega))}=M_\varepsilon\to M\;\;\mbox{as}\;\;\varepsilon\to 0^+.
$$
    Hence, from the weakly star lower semi-continuity of $L^\infty$-norm, we have
$$
    \|\tilde{u}(t)\|_\Omega\leq M\;\;\mbox{a.e.}\;\; t\in[0,T^*)
$$
    and $\tilde{u}$ is an optimal control of the problem $(TP)$.
    By the uniqueness of optimal control to problem $(TP)$ (see Proposition \ref{lemmayu3.1}), we have
$$
    \tilde{u}\equiv u^*\;\;\mbox{in}\;\; [0,T^*).
$$
    Since $\{\varepsilon_n\}_{n\in\mathbb{N}}$ was arbitrarily taken from in $\{\varepsilon\}_{\varepsilon\in(0,\varepsilon_0]}$, we have
\begin{equation}\nonumber
    u^{*,2}_{\varepsilon}\to u^*\;\;\mbox{weakly star in}\;\;L^\infty(0,T^*;L^2(\Omega))\;\;\mbox{as}\;\;\varepsilon\to0^+.
\end{equation}
    Therefore, similar to the proof of $(ii)$ in Theorem \ref{ithyu1.1} and by the results of parts $(i)$ and $(iii)$
    in this theorem, we can deduce the result of part $(iv)$.

\par
    Finally, we show the part $(v)$. Given a fixed $\eta\in(0,T^*)$. By the formula of $u^*$ (see (\ref{yu0.01})) and the result of part $(ii)$, for each $t\in[0,T^*-\eta]$, we have
\begin{eqnarray}\label{yu-b-4.24}
    &\;&\|u^{*,2}_\varepsilon(t)-u^*(t)\|_\Omega\nonumber\\
    &=&\left\|M_\varepsilon\frac{\chi_\omega\varphi^\varepsilon
    (t;\hat{\varphi}_{T^*}^\varepsilon,
    T^*)}{\|\varphi^\varepsilon(t;\hat{\varphi}_{T^*}^\varepsilon,
    T^*)\|_\omega}-M\frac{\chi_\omega\varphi(t;\hat{\varphi}_{T^*},T^*)}
    {\|\varphi(t;\hat{\varphi}_{T^*},T^*)\|_\omega}\right\|_\Omega\nonumber\\
    &\leq&|M_\varepsilon-M|+M\left\|\frac{\chi_\omega\varphi^\varepsilon
    (t;\hat{\varphi}_{T^*}^\varepsilon,
    T^*)}{\|\varphi^\varepsilon(t;\hat{\varphi}_{T^*}^\varepsilon,
    T^*)\|_\omega}-\frac{\chi_\omega\varphi(t;\hat{\varphi}_{T^*},T^*)}
    {\|\varphi(t;\hat{\varphi}_{T^*},T^*)\|_\omega}\right\|_\Omega.
\end{eqnarray}
    Meanwhile, by the result of Corollary \ref{corollaryyu1.1-1}, we have
$$
    \hat{\varphi}_{T^*}^\varepsilon\to \hat{\varphi}_{T^*}\;\;\mbox{strongly in}\;\; L^2(\Omega)\;\;\mbox{as}\;\; \varepsilon\to 0^+.
$$
    Hence, by using the similar method as that used in the proof of part $(iii)$ of Theorem \ref{ithyu1.1}, we can get
\begin{equation}\label{yu-b-4.25}
    \sup_{t\in[0,T^*-\eta]}\left\|\frac{\chi_\omega\varphi^\varepsilon
    (t;\hat{\varphi}_{T^*}^\varepsilon,
    T^*)}{\|\varphi^\varepsilon(t;\hat{\varphi}_{T^*}^\varepsilon,
    T^*)\|_\omega}-\frac{\chi_\omega\varphi(t;\hat{\varphi}_{T^*},T^*)}
    {\|\varphi(t;\hat{\varphi}_{T^*},T^*)\|_\omega}\right\|_\Omega\to 0\;\;\mbox{as}\;\;\varepsilon\to 0^+.
\end{equation}
    It follows from the result of part $(iii)$ (i.e., $M_\varepsilon\to M$ as $\varepsilon\to 0^+$) that
\begin{equation}
    \sup_{t\in[0,T^*-\eta]}\|u^{*,2}_\varepsilon(t)-u^*(t)\|_\Omega\to 0\;\;\mbox{as}\;\;
    \varepsilon\to 0^+.
\end{equation}
    The proof is completed.
\end{proof}
\section{Further comments}
\begin{enumerate}
\item
   When the target set is the origin of the state space $L^2(\Omega)$ instead of the closed ball $\overline{B_K(0)}$ in our study, it is extremely  hard for us to
   show the same results obtained in this paper. The reason is as follows: In the case that the target set is $\{0\}$ in $L^2(\Omega)$, the corresponding functional $J^{T_\varepsilon}_\varepsilon$ reads as:
   $$
   J^{T_\varepsilon}_\varepsilon(\varphi^\varepsilon_{T_\varepsilon})=\frac{1}{2}\left(\int_0^{T_\varepsilon}\|\varphi^\varepsilon
   (t;\varphi^\varepsilon_{T_\varepsilon},T_\varepsilon)\|_\omega dt\right)^2+\langle y_0,\varphi^\varepsilon(0;\varphi^\varepsilon_{T_\varepsilon},T_\varepsilon)\rangle_\Omega,\;\;\varphi^\varepsilon_{T_\varepsilon}\in L^2(\Omega),
   $$
   where $\varphi^\varepsilon(\cdot;\varphi^\varepsilon_{T_\varepsilon},T_\varepsilon)$ is the solution to equation (\ref{104-1}) with the initial time $T_\varepsilon>0$ and the initial data $\varphi^\varepsilon_{T_\varepsilon}\in L^2(\Omega)$. It has no minimizer in $L^2(\Omega)$ (at least, we do not know how to show it). This functional has a unique minimizer $\hat{\varphi}^\varepsilon_{T_\varepsilon}$ in a space $X_{T_\varepsilon}^\varepsilon$ which is closure of $L^2(\Omega)$ in a suitable norm (see Section 3 in \cite{b4}). Since  $X_{T_\varepsilon}^\varepsilon$  may be different for different $\varepsilon$, we do not know in which space  $\{\hat{\varphi}^\varepsilon_{T_\varepsilon}\}_{\varepsilon>0}$ stay and are bounded (see the proof of Theorem \ref{theoremyu1.1}).

\item
It should be interesting to improve  the convergence of the optimal control in part $(iii)$ of Theorem 1.1, more precisely, to derive
$$
u^{*,1}_\varepsilon\rightarrow u^*\;\;\mbox{in}\;\;L^\infty(0,T^*; L^2(\Omega))\;\;\mbox{as}\;\;\varepsilon\to 0.
$$
The same can be said about the convergence in part $(v)$ of Theorem 1.2. Unfortunately, by our method, we cannot get the above convergence. The reason is  as follows: We do not know if  it holds that
     $\chi_\omega\hat{\varphi}_{T^*}\neq 0$ in $\Omega$ (see the proof of (\ref{yu4.44})).
\item
    It is natural to ask if the main theorems still hold for the heat equations with space-time potentials. Indeed, after carefully checking the proofs of main results in this paper,
    we observe that Theorem 1.1 and Theorem 1.2 hold when the controlled system has the following properties:
\vskip 5pt
   \noindent $(a)$ The energy decay property of the solution to the controlled equation with the null control;
\vskip 5pt
   \noindent $(b)$ The explicit observability estimate, i.e., for each $T>0$, it holds that
  \begin{eqnarray}\nonumber
    &\;&\|\varphi(0;\varphi_T,
    T)\|_\Omega\leq
    \exp\biggl[C_0\bigg(1+\frac{1}{T}+T+
    (T^{\frac{1}{2}}+T)
    \|a\|_{L^\infty(\Omega\times(0,T))}\nonumber\\
    &\;&~~~~~~~+\|a\|_{L^\infty(\Omega\times(0,T))}^{\frac{2}{3}}\biggl)\biggl]
    \int_0^{T}\|\varphi
    (t;\varphi_T,T)
    \|_{\omega}dt,
\end{eqnarray}
  for all solutions $\varphi(\cdot;\varphi_T,T)$ to the adjoint equation:
\begin{equation}\label{yu-5-1}
\begin{cases}
    \varphi_t+\triangle\varphi
    +a \varphi=0
    &\mbox{in}\;\;\Omega\times(0,T),\\
    \varphi=0&\mbox{on}\;\;
    \partial\Omega\times(0,T),\\
    \varphi(T)
    =\varphi_{T}\in L^2(\Omega),
\end{cases}
\end{equation}
    where $a\in L^\infty(\Omega\times\mathbb{R}^+)$ and $C_0$ is a positive constant  depending only on
    $\Omega$ and $\omega$.

\vskip 5pt
  \noindent $(c)$ The unique continuation property at one time, i.e., if there exists a $t\in[0,T)$ such that $\|\varphi(t;\varphi_T,T)\|_\omega=0$, then
  $\|\varphi(\cdot;\varphi_T,T)\|_\Omega\equiv 0$ in $[0,T]$. Here $\varphi(\cdot;\varphi_T,T)$ is the solution to equation (\ref{yu-5-1}) and $T>0$.

\vskip 5pt
  \noindent $(d)$ The bang-bang property, i.e., the optimal controls to $(TP)$, $(TP_1^\varepsilon)$ and $(TP_2^\varepsilon)$ satisfy
$$
    \|u^*(t)\|=M,\;\;\mbox{for any}\;\;t\in[0,T^*),\;\;\;\;\|u^{*,1}_\varepsilon(t)\|=M,\;\;\mbox{for any}\;\;t\in[0,T^{*,1}_\varepsilon),
$$
    and
$$
    \|u^{*,2}_\varepsilon(t)\|=M_\varepsilon,\;\;\mbox{for any}\;\;t\in[0,T^{*,2}_\varepsilon).
$$
\vskip 5pt
  \noindent $(e)$ The equivalence of time and norm optimal control problems.
\vskip 5pt
  \noindent $(f)$ The explicit expression of the optimal control to time optimal control, i.e.,
$$
    u^*(t)=M\frac{\chi_\omega\varphi(t;\hat{\varphi}_{T^*},T^*)}
    {\|\varphi(t;\hat{\varphi}_{T^*},T^*)\|_\omega}\;\;\mbox{for any}\;\;t\in[0,T^*),
$$
    where $\hat{\varphi}_{T^*}$ is the minimizer of $J^{T^*}$ defined by (\ref{101}) and $\varphi(\cdot;\hat{\varphi}_{T^*},T^*)$ is the solution to equation (\ref{yu-5-1}) with the initial time $T^*>0$ and the initial data $\hat{\varphi}_{T^*}\in L^2(\Omega)$. The same can be said about the optimal controls $u^{*,1}_\varepsilon$ and $u^{*,2}_\varepsilon$.
  \vskip 5 pt
   Now we consider $(TP)$, $(TP^\varepsilon_1)$ and $(TP^\varepsilon_2)$ corresponding to equation (\ref{equationyu2.01}) and (\ref{equationyu2.02}), where $a =a(x,t)\in L^\infty(\Omega\times\mathbb{R}^+)$ and $a_\varepsilon=a_\varepsilon(x,t)\in L^\infty(\Omega\times\mathbb{R}^+)$ satisfy
   \vskip 5pt

   \noindent  $(H_1')$   $\|a_\varepsilon-a\|_{L^\infty(\Omega\times\mathbb{R}^+)}\to
    0$ as $\varepsilon\to 0^+$.

 \vskip 5pt

    \noindent  $(H_3')$  Either  $\|a\|_{L^\infty(\Omega\times\mathbb{R}^+)}<\lambda_1$
    or $a(x,t)\leq0$ for any
    $(x,t)\in\Omega\times\mathbb{R}^+$,
    where $\lambda_1>0$ is the first eigenvalue to the operator
    $-\triangle$ with the domain  $D(\triangle)=H_0^1(\Omega)\cap H^2(\Omega)$.

    \vskip 5pt

    The condition $(a)$ is implied by the assumption $(H_3')$; The condition $(b)$ is given by Proposition 3.2 in \cite{b1}; The condition $(c)$ is given by \cite{b31} (see also \cite{b30}); The condition $(d)$  can be derived from the Pontryagin maximum principle (see \cite{b10} or Theorem 4.1 of Chapter 7 in \cite{b32}) and the unique continuation property $(c)$; The condition $(e)$ can be derived from the above conditions $(a)$-$(d)$, via the almost same way in \cite{b4}; The condition $(f)$ follows from Conditions $(c)$ and $(e)$. Hence, Theorem 1.1 and Theorem 1.2 still hold when $a$ and $a_\varepsilon$ are space-time dependent and hold $(H_1')$, $(H_2)$ and $(H_3')$.

\end{enumerate}
\vskip 8pt
\par
    \textbf{Acknowledgment.} The author would like to thank Professor Gengsheng Wang deeply for the encouragement and suggestions. Also, the author gratefully
    acknowledges the anonymous referees for the suggestions which led to this improved version.


\begin{thebibliography}{99}
\bibitem{b2} V. Barbu, Nonlinear differential equations of monotone
    types in Banach spaces, Springer, New York, 2010.
\bibitem{b15} M. Bardi, Boundary value problem for the minimum-time
    function, SIAM J. Control Optim. 27 (1989), 776-785.
\bibitem{b16} O. Carja, The minimal time function in infinite dimensions,
     SIAM J. Control Optim. 31 (1993), 1103-1114.
\bibitem{b20} O. Carja, On the minimum time function and
    the minimum energy problem; a nonlinear case, Systems Control
    Lett. 55 (2006), 543-548.
\bibitem{b11} T. Duyckaerts, X. Zhang and E. Zuazua, On the
    optimality of the observability inequalities for parabolic and
    hyperbolic systems with potentials, Ann. Inst. H. Poincar\'{e} (C)
    Non Linear Anal. 25 (2008), 1-41.
\bibitem{b6} C. Fabre, J.-P. Puel and E. Zuazua, Approximate
    controllability of the semilinear heat equation, Proc. Royal Soc.
    Edinburgh Sect. A 125 (1995),  31-61.
\bibitem{b24} H. O. Fattorini, Infinite Dimensional Linear Control Systems, The Time Optimal and Norm Optimal Problem, in North-Holland Math. Stud. 201, Elsevier, New York, 2005.
\bibitem{b1} E. Fern\'{a}ndez-Cara and E. Zuazua, Null and
    approximate controllability for
    weakly blowing up semilinear heat equations, Annales de l'Institut Henri Poincar\'{e}, Analyse Non Lin\'{e}aire, 17
    (2000), 583-616.
\bibitem{b12} A. V. Fursikov and O. Y. Imanuvilov, Controllability
    of Evolution Equations, in Lecture Notes Ser. 34, Seoul National
    University, Research Institute of Mathematics, Global Analysis
    Research Center, Seoul, 1996.
\bibitem{b17} F. Gozzi and P. Loreti, Regularity of the minimum
    time function and minimum
    energy problems: The linear case, SIAM J. Control Optim. 37 (1999),
    1195-1221.
\bibitem{b32}  X. Li and J. Yong, Optimal Control Theory for Infinite Dimensional Systems, Birkhauser Boston, Boston, MA, 1995.
\bibitem{b5} F. H. Lin, A uniqueness theorem for the parabolic
    equation, Comm. Pure Appl. Math. 43 (1990), 127-136.
\bibitem{b25} A. Pazy, Semigroups of Linear Operators and
    Applications to Partial Differential Equations, Springer-Verlag, New York, 1983.
\bibitem{b30} K. D. Phung and G. Wang, Quantitative unqiue continuation for the semilinear heat equation in a convex domian, J. Funct. Anal. 259 (2010), 1230-1247.
\bibitem{b7} K. D. Phung, G. Wang and X. Zhang, On the existence
    of time optimal controls for linear
    evolution equations, Discrete Contin. Dyn. Syst. Ser. B  8 (2007),
    925-941.
\bibitem{b31} K. D. Phung, L. Wang and C. Zhang, Bang-Bang property for time optimal control of semilinear heat equation, Annales de l'Institut Henri Poincar\'{e}, Analyse Non Lin\'{e}aire, to appear.
\bibitem{b19} F. Rampazzo and C. Sartori, The minimum time
    function with unbounded controls, J. Math. Systems, Estimation. and Control
    8 (1998), 1-34.
\bibitem{b21} P. Soravia, H\"{o}lder continuity of the minimum-time
    function for $C^1$-manifold targets, J. Optim. Theory Appl. 75 (2)
    (1992), 401-421.
\bibitem{b18} V. M. Veliov, Lipschitz continuity of the value function
    in optimal control, J. Optim. Theory Appl. 94 (2) (1997),
    335-363.
\bibitem{b8} G. Wang, The existence of time optimal control of
    semilinear parabolic equations, Systems Control Lett. 53
    (2004), 171-175.
\bibitem{b10} G. Wang and L. Wang, The Bang-Bang principle of time
    optimal controls for the heat equation with internal controls,
    Systems Control Lett. 56 (2007), 709-713.
\bibitem{b9} G. S. Wang, $L^\infty$-null controllability for the
    heat equation and its consequences for the time optimal control
    problem, SIAM J. Control Optim. 47 (4) (2008), 1701-1720.
\bibitem{b22} G. Wang and Y. Xu, Equivalence of three different
    kinds of optimal control problems for
    heat equations and its applications, SIAM J. Control Optim. 51 (2) (2013), 848-880.
\bibitem{b4} G. Wang and E. Zuazua, On the equivalence of minimal
    time and minimal norm controls for internally controlled heat
    equations, SIAM J. Control Optim. 50 (5) (2012), 2938-2958.
\bibitem{b3} E. Zuazua, Controllability and observability of partial
    differental equations: Some results and open problems, Handbook of
    Differential equations: Evolutionary Differential Equations, Vol. 3,
    Elsevier Science, (2006), 527-621.

\end{thebibliography}
\end{document}